\documentclass[onefignum,onetabnum]{siamart190516}

\ifpdf
\hypersetup{
	pdftitle={Inexact Proximal Newton methods in Hilbert Spaces},
	pdfauthor={B. Pötzl, A. Schiela, and P. Jaap}
}
\fi

\usepackage[utf8]{inputenc}
\usepackage[english]{babel}
\usepackage[T1]{fontenc}
\usepackage{amsmath}
\usepackage{amssymb}
\usepackage{wasysym}
\usepackage{bbm}
\usepackage{mathtools}
\usepackage{hyperref}
\usepackage{todonotes}
\usepackage{multirow}
\usepackage{algorithmic}
\usepackage{subcaption}

\usepackage{float}
\usepackage{stackengine}
\usepackage{diagbox}
\usepackage{fancyhdr}

\usepackage{geometry}
\usepackage{graphicx}
\graphicspath{}

\title{Inexact Proximal Newton methods in Hilbert Spaces}

\newsiamremark{remark}{Remark}
\newsiamremark{hypothesis}{Hypothesis}
\crefname{hypothesis}{Hypothesis}{Hypotheses}
\newsiamthm{claim}{Claim}

\headers{Inexact Proximal Newton methods in Hilbert Spaces}{B. Pötzl, A. Schiela, and P. Jaap}

\title{Inexact Proximal Newton methods in Hilbert Spaces\thanks{Submitted to the editors 04/26/2022.
		\funding{This work was funded by the DFG SPP 1962: Non-smooth and Complementarity-based Distributed Parameter Systems -- Simulation and Hierarchical Optimization; Project number: SCHI 1379/6-1}}}

\author{Bastian Pötzl\thanks{Department of Applied Mathematics, Universität Bayreuth, Bayreuth, Germany 
		(\email{bastian.poetzl@uni-bayreuth.de}, \email{anton.schiela@uni-bayreuth.de}).}
	\and Anton Schiela\footnotemark[2]
	\and Patrick Jaap\thanks{Institute of Numerical Mathematics, TU Dresden, Dresden, Germany
		(\email{patrick.jaap@tu-dresden.de}).}}

\usepackage{amsopn}

\date{\today}

\newcommand{\NN}{\mathbb{N}}
\newcommand{\RR}{\mathbb{R}}
\newcommand{\de}{\mathrm{d}}
\newcommand{\norm}[2]{\bigl\| #1 \bigl\|_{#2}}

\newcommand{\argmin}{\operatornamewithlimits{argmin}}
\newcommand{\p}{\mathcal{P}}

\newcommand{\Lin}{\mathcal{L}(X,X^*)}
\newcommand{\keins}{\kappa_1}
\newcommand{\kzwei}{\kappa_2}
\newcommand{\ksum}{\kappa_1 + \kappa_2}
\newcommand{\riesz}{\mathcal R}
\newcommand\xrowht[2][0]{\addstackgap[.5\dimexpr#2\relax]{\vphantom{#1}}}
\newcommand{\omax}{\tilde \omega_{\mathrm{max}}}

\begin{document}

\maketitle

\begin{abstract}
  We consider Proximal Newton methods with an inexact computation of update steps. To this end, we introduce two inexactness criteria which characterize sufficient accuracy of these update step and with the aid of these investigate global convergence and local acceleration of our method. The inexactness criteria are designed to be adequate for the Hilbert space framework we find ourselves in while traditional inexactness criteria from smooth Newton or finite dimensional Proximal Newton methods appear to be inefficient in this scenario. The performance of the method and its gain in effectiveness in contrast to the exact case are showcased considering a simple model problem in function space.
\end{abstract}



\section{Introduction} \label{sec:intro}

In the present work we extend the idea of Proximal Newton methods in Hilbert spaces as presented in \cite{Poetzl2022} to admit an inexact computation of update steps by solving the respective subproblem only up to prescribed accuracy. We consider the composite minimization problem
\begin{align}\label{eq:prob} 
\min_{x \in X} F(x) \coloneqq f(x) + g(x)
\end{align}
on some real Hilbert space $(X,\langle\cdot,\cdot\rangle_X)$ where $f: X \rightarrow \RR$ is assumed to be smooth in some adequate sense and $g: X \rightarrow \RR$ is possibly not. We pay particular attention to the infinite-dimensionality of the underlying Hilbert spaces and thus develop inexactness criteria for update step computation that are sufficiently easy to evaluate, help us preserve convergence properties of the exact case as considered in \cite{Poetzl2022} and reduce the computational effort significantly.

For an overview of the development of Proximal Newton methods themselves consider \cite{Poetzl2022}. Here, we want to focus on the realization of the inexactness aspect and consider corresponding most recent literature in this introductory section. The use of gradient-like inexactness criteria which can be seen as the direct generalization of the one for classical smooth Newton methods in \cite{Dembo1982} is quite common, cf. \cite{Lee2014, Byrd2015, Kanzow2020}. 

In \cite{Lee2014} additional knowledge of bounds on the second-order bilinear forms as well as the Lipschitz constant of $f'$ is necessary and only local convergence has been investigated in the inexact case. Globalization of the ensuing method has been achieved in \cite{Kanzow2020} by using a Proximal Gradient substitute step in case the inexactly computed second order step does not suffice a sufficient decrease criterion or the step computation subproblem is ill-formed due to non-convexity which thus can be overcome as well. In \cite{Byrd2015} the particular case of $L_1$-regularization for machine learning applications has been considered and thus the inexactness criterion has further been specified and also here enhanced with a decrease criterion in the quadratic approximation of the composite objective function. The latter has then been tightened in order to achieve local acceleration. 

Another approach to inexactness criteria is measuring the residual within the step computation subproblem. In \cite{Li2016}, where objective functions consisting of the sum of a thrice continuously differentiable smooth part and a self-concordant non-smooth part have been considered, the residual vector within optimality conditions for update computation is supposed to be bounded in norm with respect to the already computed inexact step. However, the residual can also be measured via functional descent in the quadratic approximation of the composite objective $F$, cf. \cite{Lee2019, Scheinberg2016}. While in \cite{Lee2019} the second order model decrease bound against its optimal value has not directly been tested but simply assumed to hold after a finite (and fixed) number of subproblem solver iterations, the authors in \cite{Scheinberg2016} have taken the structure of their randomized coordinate descent subproblem solver into account and also have given quadratic bounds for the prefactor constant within their model descent estimate in order to obtain sufficient convergence results.

All of the above works have in common that they depend on the finite dimensional structure of the underlying Euclidean space. In particular, the efficient computation of proximal gradients, required for the evaluation of inexactness cirteria, relies on the diagonal structure of the underlying scalar product $\langle \cdot,\cdot\rangle_X$, which is usually not present in (discretized) function spaces, as for example, Sobolev spaces. Moreover, all current approaches consider fixed search directions which are then scaled by some step length parameter. 

Our contributions beyond their work can be summarized as follows: 
Most importantly, we replace the Euclidean space setting with a Hilbert space one in order to rigorously allow function space applications of our method. In particular, we are interested in the important case where $X$ is a Sobolev space. Then, a diagonal approximation of $\langle \cdot,\cdot \rangle_X$ after discretization would lead to proximal operators that suffer from mesh-dependent condition numbers. For the efficient computation of proximal steps we thus take advantage of a non-smooth multigrid method. Specifically, we use a Truncated Non-smooth Newton Multigrid Method (TNNMG), cf. \cite{Graeser2018} in our numerical implementation. Consequently, our inexactness criteria need to be constructed in such a way that their evaluation is efficient in this context. Existing criteria can only be employed efficiently, if $\langle \cdot,\cdot \rangle_X$ enjoys a diagonal structure. 


Additionally, ellipticity of the bilinear forms for forming quadratic approximations of our objective functional as well as convexity of the non-smooth part $g$ has often crucial been in literature. We drop these prerequisites and use a less restrictive framework of convexity assumptions for the composite objective function $F$. Finally, we do not demand second order differentiability with Lipschitz-continuous second order derivative of the smooth part $f$ but instead settle for adequate semi-smoothness assumptions. 


Let us now give the precise set of assumptions in which we will discuss the convergence properties of inexact Proximal Newton methods. As pointed out beforehand, we find ourselves in a real Hilbert space $(X,\langle\cdot,\cdot\rangle_X)$ with corresponding norm $\|v\|_X=\sqrt{\langle v,v\rangle_X}$ and dual space $X^*$. This choice of $X$ also provides us with the Riesz-Isomorphism $\mathcal R : X \to X^*$, defined by $\mathcal Rx=\langle x,\cdot\rangle_X$, which satisfies $\norm{\mathcal{R} x}{X^*} = \norm{x}{X}$ for every $x\in X$. Since $\mathcal R$ is non-trivial in general, we will not identify $X$ and $X^*$. 

The smooth part of our objective functional $f:X \rightarrow \RR$ is assumed to be continuously differentiable with Lipschitz-continuous derivative $f':X\to X^*$, i.e., we can find some constant $L_f > 0$ such that for every $x,y \in X$ we obtain the estimate
\begin{align} \label{eq:Lipschitz}
\norm{f'(x) - f'(y)}{X^*} \leq L_f \norm{x-y}{X} \, .
\end{align}

As mentioned beforehand, we will use the base algorithm from \cite{Poetzl2022} as our point of departure. This means that we consider a variation of the Proximal Newton method which is globalized by an additional norm term within the subproblem for step computation. As a consequence, the latter reads
\begin{align} \label{eq:dampedstep}
\Delta x (\omega) \coloneqq \argmin_{\delta x \in X} \lambda_{x,\omega}(\delta x)
\end{align}
where the regularized second order decrease model $\lambda_{x,\omega}:X \to \RR$ is given by
\begin{align*}
\lambda_{x,\omega}(\delta x) \coloneqq f'(x)\delta x + \frac 12 H_x (\delta x, \delta x) + \frac{\omega}{2}\norm{\delta x}{X}^2 + g(x+\delta x) - g(x) \, .
\end{align*}
The updated iterate then takes the form $x_+ (\omega) \coloneqq x + \Delta x (\omega)$.  

The second order model of the smooth part $f$ from above also has to be endowed with adequate prerequisites. Notationally identifying the linear operators $H_x \in \Lin$ with the corresponding symmetric bilinear forms $H_x : X \times X\to \RR$ we write $(H_x v)(w)=H_x(v,w)$ and abbreviate $H_x(v)^2=H_x(v,v)$. Uniform boundedness of the $H_x$ along the sequence $(x_k)$ of iterates in the form
\begin{align} \label{eq:Hxbound}
\norm{H_{x_k}}{\Lin} \leq M \quad \text{for some (uniform)} \quad M > 0 \, .
\end{align}
will also be of importance in what follows. Furthermore, along the sequence of iterates $(x_k)$ we assume a (possibly non-uniform) bound of the form
\begin{align} \label{eq:kappa1}
\forall \, k \in \NN \; \exists \, \keins \in \RR \, \forall v\in X \colon H_{x_k} (v)^2 \coloneqq H_{x_k}(v,v) \geq \keins \norm{v}{X}^2
\end{align}
which can be interpreted as an ellipticity assumption on $H_{x_k}$ in case the constant $\keins$ is positive. In this case, when considering exact (and smooth) Proximal Newton methods, where $H_x$ is given by the Hessian of $f$ at some point $x \in X$, \eqref{eq:kappa1} is equivalent to $\keins$-strong convexity of $f$.

While in a sufficiently smooth setting $H_x \coloneqq f''(x)$ is common, for most of the paper we may choose $H_x$ freely in the above framework. For fast local convergence, however, we will impose a semi-smoothness assumption, cf. \eqref{eq:semismooth}. Semi-smooth Newton methods in function space have been discussed, for example, in \cite{Ulbrich2002,Ulbrich2011,Hintermueller2002,Schiela2006}.

As far as the non-smooth part $g$ is concerned, we require lower semi-continuity as well as a bound of the form
\begin{align}\label{eq:kappa2}
g(sx+(1-s)y) \leq sg(x) + (1-s)g(y) - \frac{\kzwei}{2}s(1-s)\norm{x-y}{X}^2
\end{align}
for all $x,y \in X$ and all $s \in [0,1]$ for some $\kzwei \in \RR$. For $\kzwei > 0$ this estimate can be interpreted as $\kzwei$-strong convexity of $g$. In the latter case we can then conclude that $g$ is bounded from below, its level-sets $L_\alpha g$ are bounded for all $\alpha \in \RR$ and that their diameter shrinks to $0$ in the limit of $\alpha \to \inf_{x\in X} g$. 
Non-positivity of $\kzwei$ allows $g$ to be non-convex in a limited way. 

The theory behind Proximal Newton methods and the respective convergence properties evolve around the convexity estimates stated in \eqref{eq:kappa1} and \eqref{eq:kappa2}. We will assign particular importance to the interplay of the convexity properties of $f$ and $g$, i.e., the sum $\ksum$ will continue to play an important role over the course of the present treatise. Apparently, the update step in \eqref{eq:dampedstep} is well defined for every $\omega>0$ if $\ksum > 0$. This holds also in the case of $\ksum \leq 0$ for every $\omega > -(\ksum)$ due to the bounds stated in \eqref{eq:kappa1}, \eqref{eq:kappa2} and the strong convexity of the norm term. For this reason, we will assume $\omega > -(\ksum)$ wherever it appears. 

The above demands on $f$, $g$, $H_x$ and $\omega$ constitute the standing assumptions for the further investigation which we impose for the entirety of the paper.


Let us now briefly outline the structure of our work: In Section~\ref{sec:gradientmap} we introduce the notion of composite gradient mappings and consider some of their basic properties. Afterwards, in Section~\ref{sec:local}, we take advantage of the acquired knowledge and introduce the first inexactness criterion in order to investigate local convergence of our method as well as the influence of both damping and inexactness. Section~\ref{sec:global} then considers the globalization phase of our inexact Proximal Newton method and for this reason introduces a second inexactness criterion which compares the functional decrease of inexact updates with steps originating from a simpler subproblem. Thus, we also achieve sufficient global convergence results. In order to then benefit from local acceleration, we investigate the transition to local convergence in Section~\ref{sec:transition}. To this end, we need to ensure that close to optimal solutions also arbitrarily weakly damped update steps yield sufficient decrease. Lastly, we put our method to the test in Section~\ref{sec:numres} and display global convergence as well as local acceleration considering a simple model problem in function space. Concluding remarks can be found in Section~\ref{sec:concl}.

\section{Composite Gradient mappings and their Properties} \label{sec:gradientmap}

The main goal to keep in mind is not only to introduce the concept of inexactness to the computation of update steps of the Proximal Newton method from \cite{Poetzl2022} but also quantify the influence of damping update steps to the local convergence rate of our algorithm. 

\subsection{Definition and Representation via Proximal Mappings}
For this cause, we take advantage of the notion of regularized composite gradient mappings $G_\tau^\Phi:X \to X$ for some composite functional $\Phi: X \to \RR$ in the form $\Phi (x) \coloneqq \phi(x) + \psi(x)$ with smooth part $\phi: X \to \RR$ and non-smooth part $\psi: X \to \RR$. The aforementioned mapping is defined via 
\begin{align} \label{eq:gradientdef}
\begin{split}
G_\tau^\Phi (y) &\coloneqq -\tau \bigg[\argmin_{\delta y \in X} \phi'(y) \delta y + \frac{\tau}{2} \norm{\delta y}{X}^2 + \psi(y + \delta y) - \psi(y)\bigg]
\end{split}
\end{align}
for $y \in X$ and some regularization parameter $\tau > 0$ the assumptions on which we will specify over the course of the current section. For the derivation of useful estimates for composite gradient mappings, the so-called scaled dual proximal mapping $\p_\psi^H:X^* \to X$, defined via 
\begin{align*}
\p_\psi^H(\ell) \coloneqq \argmin_{z \in X} \psi(z) + \frac 12 H(z,z) - \ell (z)
\end{align*}
for arbitrary $\ell \in X^*$ and some symmetric bilinear form $H$ sufficing \eqref{eq:kappa1} as well as some real valued function $\psi$ satisfying \eqref{eq:kappa2} for constants $\keins,\kzwei \in \RR$ with $\ksum > 0$, will come in handy. In what is to come, we will take advantage of the following two crucial results concerning dual proximal mappings which have been stated and proven in \cite{Poetzl2022}. The first one is a general estimate for the image of such operators which generalizes the assertions of the so called second prox theorem, cf. e.g. \cite[Chapter~6.5]{Beck2017}. The second one is a Lipschitz-continuity result.
\begin{proposition}[\cite{Poetzl2022}, Proposition 2 and Corollary 1] \label{prop:scndprox}
	Let $H$ and $\psi$ satisfy the assumptions \eqref{eq:kappa1} and \eqref{eq:kappa2} with \mbox{$\ksum > 0$}. Then for any $\ell \in X^*$ the image of the corresponding proximal mapping \mbox{$u \coloneqq \p_\psi^H(\ell)$} satisfies the estimate
	\begin{align*}
		\big[ \ell - H(u) \big](\xi - u) \leq \psi(\xi) - \psi(u) - \frac{\kzwei}{2}\norm{\xi - u}{X}^2
	\end{align*}
	for all $\xi \in X$. Additionally, for all $\ell_1,\ell_2 \in X^*$ the following inequality holds:
	\begin{align*}
		\norm{\p_\psi^H(\ell_1) - \p_\psi^H(\ell_2)}{X} \leq \frac{1}{\ksum} \norm{\ell_1 - \ell_2}{X^*} \, .
	\end{align*}
\end{proposition}
With the aid of scaled proximal mappings, we can express the composite gradient mapping as
\begin{align} \label{eq:gradientprox}
G_\tau^\Phi (y) = \tau \big[ y - \p_\psi^{\tau \mathcal R} \big( \tau \mathcal R y - \phi'(y) \big) \big] \, .
\end{align}
Let us now justify the designation of $G_\tau^\Phi$ as a regularized composite gradient mapping. If we consider the smooth case of $\psi = 0$, the proximal mapping takes the form $\p_\psi^H(\ell) = H^{-1}\ell$. This fact carries over to the definition of the gradient mapping via
\begin{align*}
G_\tau^\phi (y) = \tau \big[ y - (\tau \mathcal R)^{-1} \big( \tau \mathcal R y - \phi'(y) \big) \big] = \riesz^{-1} \phi'(y)
\end{align*}
which resembles the infinite dimensional counterpart of the gradient $\nabla \phi$ in Euclidean space. Note that this consistency result holds for all $\tau > 0$. 

Another consideration which expresses the consistency between $G_\tau^F$ and some actual 'smooth' gradient of $F=f+g$ with respect to our minimization problem \eqref{eq:prob} is the following: Let then $G_\tau^F (x_*) = 0$ hold for some $x_* \in X$ and $\tau \geq 0$. This is equivalent to the fixed point equation $x_* = \p_g^{\tau \mathcal R} \big( \tau \mathcal R x_* - f'(x_*) \big)$ which can then again be transformed to $-f'(x_*) \in \partial_F g(x_*)$ in $X^*$. Consequently, we recognize that the composite gradient mapping is zero if and only if we evaluate it at critical points of the underlying minimization problem \eqref{eq:prob}.

\subsection{Key Properties and Auxiliary Estimates}
For now, let us derive some key properties of the composite gradient mappings which will come in handy as we quantify the influence of both inexactness and damping to local convergence rates of our algorithm. 

Before departing on this endeavor we introduce the modified quadratic model $\hat F_{x,\omega}: X \to \RR$ of the composite objective functional $F$ around $x\in X$ with regularization parameter $\omega$ via
\begin{align} \label{eq:dampedmodel}
\begin{split}
	\hat F_{x,\omega} (y) &\coloneqq F(x) + \lambda_{x,\omega}(y-x) = f(x) + f'(x)(y-x) + \frac 12 H_x(y - x)^2 + g(y) + \frac{\omega}{2}\norm{y-x}{X}^2 \, .
\end{split}
\end{align}
The corresponding composite gradient mapping $G^{\hat F_{x,\omega}}_\tau$ will play an important role. In that regard, we note that in the framework of the definition of the gradient mapping in \eqref{eq:gradientdef} we thus have $\Phi = \hat F_{x,\omega} = \hat \phi + \hat \psi$ with
\begin{align} \label{eq:modelsplit}
\hat \phi (y) =  f(x) + f'(x)(y-x) + \frac 12 \big(H_x + \omega \riesz \big)(y - x)^2 \quad , \quad \hat \psi (y) = g(y)
\end{align}
and thereby $\hat \phi '(y) = f'(x) + \big(H_x + \omega \riesz \big)(y - x)$ for any $y \in X$. The following lemma provides us with helpful estimates for the norm difference of composite gradient mappings both from above and below.

\begin{lemma} \label{lem:GFxyz}
	For every $x,y,z \in X$ and the choice $\tau \coloneqq \omega + \frac{1}{2}\big( \norm{H_x}{\Lin}+\keins \big)$, the regularized composite gradient mapping suffices the estimate
	\begin{align} \label{eq:GFxyz}
		\tau \big( 1 - \mathcal H \big) \norm{y-z}{X} \leq \norm{G_{\tau}^{\hat F_{x,\omega}}(y) - G_{\tau}^{\hat F_{x,\omega}}(z)}{X} \leq \tau \big( 1 + \mathcal H \big) \norm{y-z}{X}
	\end{align}
	where we abbreviated $\mathcal H \coloneqq \frac{\norm{H_x}{\Lin}-\keins}{2(\tau + \kzwei)}$ \, .
\end{lemma}
\begin{proof}
	As we insert the characterizations of the respective regularized composite gradient mappings as in \eqref{eq:gradientprox}, we perceive that we can represent their norm difference via
	\begin{align*}
		\begin{split}
			\norm{G_{\tau}^{\hat F_{x,\omega}}(y) - G_{\tau}^{\hat F_{x,\omega}}(z)}{X} = \tau \norm{ (y-z)- \big(\p_y - \p_z\big) }{X}
		\end{split}
	\end{align*}
	with abbreviations $\mathcal P_\xi \coloneqq \p_g^{\tau \riesz}\big( \tau \riesz \xi - \big[ f'(x) + \big(H_x + \omega \riesz\big)(\xi - x)\big] \big)$ for $\xi \in \{y,z\}$. This provides us with the bounds
	\begin{align*}
		\tau \big( \norm{y-z}{X} - \norm{\p_y - \p_z}{X} \big) \leq \norm{G_{\tau}^{\hat F_{x,\omega}}(y) - G_{\tau}^{\hat F_{x,\omega}}(z)}{X} \leq \tau \big( \norm{y-z}{X} + \norm{\p_y - \p_z}{X} \big)
	\end{align*}
	from above and below for the norm difference of gradient mappings. This shows that for the proof of \eqref{eq:GFxyz} it suffices to verify
	\begin{align} \label{eq:GFxyzproof}
		\norm{\p_y - \p_z}{X} \leq \mathcal H \norm{y-z}{X} = \frac{\norm{H_x}{\Lin}-\keins}{2(\tau + \kzwei)}\norm{y-z}{X} \, .
	\end{align}
	The Lipschitz result from Proposition~\ref{prop:scndprox} allows us to establish the following estimate for the norm difference of proximal mapping images in relation to their arguments:
	\begin{align}
		\begin{split}
			\norm{\p_y - \p_z}{X} &\leq \frac{1}{\tau + \kzwei} \norm{\tau \riesz y - \big(H_x + \omega \riesz\big)(y - x) - \big( \tau \riesz z - \big(H_x + \omega \riesz\big)(z - x)\big)}{X^*} \\
			&= \frac{1}{\tau + \kzwei} \norm{\big((\tau - \omega) \riesz - H_x\big)(y-z)}{X^*} \leq \frac{\norm{(\tau - \omega) \riesz - H_x}{\Lin}}{\tau + \kzwei} \norm{y-z}{X} \, . \label{eq:GFkxylin}
		\end{split}
	\end{align}
	Let us now pay particular attention to the $\Lin$-norm difference in the prefactor above. On the one hand, for any $\tau > - \kzwei$, we can estimate it by 
	\begin{align*}
		\norm{(\tau - \omega)\riesz - H_x}{\Lin} \leq |\tau - \omega| + \norm{H_x}{\Lin} \, .
	\end{align*}
	Nevertheless, with further assumptions on the gradient mapping regularization parameter $\tau$ we can deduce a better bound. To this end, we define $\lambda:=\tau -\omega$ and choose $\lambda_{\mathrm{opt}}$ such that $\norm{\lambda\riesz - H_x}{\Lin}$ is minimal. It is easy to see that the eigenvalues of the self-adjoint operator $H_x^\tau \coloneqq \riesz^{-1}(\lambda\riesz - H_x)$ lie in the interval $\big[ \lambda-\norm{H_x}{\Lin},\lambda -\keins  \big]$.
	
	In order to now minimize the norm of $H_x^\tau$, we recognize that it equals the spectral radius of $H_x^\tau$ and thus want to establish a symmetrical interval where eigenvalues can be located. This 
	yields the choice $\lambda_{\mathrm{opt}} \coloneqq  \frac{1}{2}\big( \norm{H_x}{\Lin} + \keins \big)$. In particular, this implies
	\begin{align*}
		\tau \coloneqq \omega+\lambda_{\mathrm{opt}}=\omega + \frac{1}{2}\big( \norm{H_x}{\Lin} + \keins \big) \geq \omega + \frac{|\keins|+\keins}{2} \geq \omega + \keins > -\kzwei
	\end{align*}
	by our choice of $\omega$ and consequently 
	\begin{align*}
		\norm{(\tau - \omega) \riesz - H_x}{\Lin} = \norm{H_x^\tau}{\mathcal L(X,X)} = \norm{H_x}{\Lin} - \lambda_{\mathrm{opt}} = \frac{1}{2}\big( \norm{H_x}{\Lin} - \keins \big) \, .
	\end{align*}
	Inserting this into the above estimate \eqref{eq:GFkxylin}, we obtain \eqref{eq:GFxyzproof} which completes the proof.
\end{proof}

For the next result, we take advantage of the solution property of exactly computed update steps from \eqref{eq:dampedstep}.

\begin{proposition} \label{prop:gradprop}
	Let $\Delta x (\omega)$ be an exactly computed update step as in \eqref{eq:dampedstep} at some $x \in X$. Then, for any $\tau > -\kzwei$ the following identity holds:
	\begin{align} \label{eq:GFdeltaxzero}
	G_\tau^{\hat F_{x,\omega}}\big( x + \Delta x (\omega) \big) = 0 \, .
	\end{align}
\end{proposition}
\begin{proof}
	We consider the minimization problem within brackets in the definition of the regularized composite gradient mapping in \eqref{eq:gradientdef}. Here, we have to insert the derivative $\phi'$ of the smooth part of the regularized model $\hat F_{x,\omega}$ as in \eqref{eq:modelsplit} evaluated at $y = x + \Delta x (\omega)$ which yields 
	\begin{align} \label{eq:argmintau}
	\argmin_{\delta x \in X} \big[ f'(x) + \big( H_x + \omega \riesz \big)\Delta x (\omega) \big]\delta x + \frac \tau 2 \norm{\delta x}{X}^2 + g\big( x + \Delta x (\omega) + \delta x \big) - g\big( x + \Delta x (\omega)\big) \, .
	\end{align}
%
	By strong convexity of the objective function for $\tau > -\kzwei$, the above minimization problem has a unique solution $\delta \bar x \in X$. By first order optimality conditions, this solution then satisfies the dual space inclusion
	\begin{align} \label{eq:fooctau}
	0 \in f'(x) + \big( H_x + \omega \riesz \big)\Delta x (\omega) + \partial_F g\big( x + \Delta x (\omega) + \delta \bar x \big) + \tau \riesz \delta \bar x 
	\end{align}
	for the Fr\'echet-subdifferential $\partial_F g$. Note here that the exactly computed update step $\Delta x (\omega)$ as a solution of \eqref{eq:dampedstep} suffices
	\begin{align*}
	0 \in f'(x) + \big( H_x + \omega \riesz \big)\Delta x (\omega) + \partial_F g\big( x + \Delta x (\omega)\big)
	\end{align*}
	which directly yields that $\delta \bar x = 0$ satisfies \eqref{eq:fooctau} and is thereby the unique solution of \eqref{eq:argmintau}. This completes the proof of \eqref{eq:GFdeltaxzero}.
\end{proof}	

Let us now consider the difference of gradient mappings of the objective function $F$ and its modified second order model $\hat F_{x,\omega}$ at optimal solutions $x_*$ of problem \eqref{eq:prob}.

For the following we require $f'$ to be semi-smooth near an optimal solution $x_*$ of our problem \eqref{eq:prob} with respect to $H_x$, i.e., that the following approximation property holds:
\begin{align} \label{eq:semismooth}
\norm{f'(x_*) - f'(x) - H_{x}(x_* - x)}{X^*} = o\big( \norm{x-x_*}{X} \big) \, .
\end{align}
Adequate definitions of $H_x$ can be given via a so-called Newton derivative from $\partial_N f'(x)$, also known as the generalized differential $\partial^* f'(x)$ for Lipschitz-continuous operators in finite dimensions, and for corresponding superposition operators, cf. \cite[Chapter~3.2]{Ulbrich2011}.
\begin{lemma} \label{lem:GFGFk}
	Let the semi-smoothness assumption \eqref{eq:semismooth} hold near an optimal solution $x^* \in X$. Then, the regularized composite gradient mapping satisfies the following estimate for each $\tau > -\kzwei$ and $x \in X$:
	\begin{align*}
	\norm{G_{\tau}^F(x_*) - G_{\tau}^{\hat F_{x,\omega}}(x_*)}{X} \leq o\big( \norm{x_* - x}{X} \big) + \frac{\tau \, \omega}{\tau + \kzwei} \norm{x_* - x}{X} \, .
	\end{align*}
\end{lemma}
\begin{proof}
	The proof here follows immediately by the characterization of the regularized composite gradient mapping as in \eqref{eq:gradientprox} and the semi-smoothness of $f'$ according to \eqref{eq:semismooth}. To go into detail, by Proposition~\ref{prop:scndprox} we have
	\begin{align*}
		\norm{G_{\tau}^F(x_*) &- G_{\tau}^{\hat F_{x,\omega}}(x_*)}{X} \\
		&= \tau \norm{\p_g^{\tau \riesz}\big( \tau \riesz x_* - f'(x_*) \big) - \p_g^{\tau \riesz}\big( \tau \riesz x_* - \big[ f'(x) + \big(H_x + \omega \riesz\big)(x_* - x)\big]  \big)}{X} \\
		&\leq \frac{\tau}{\tau + \kzwei} \norm{\big( \tau \riesz x_* - f'(x_*) \big) - \big( \tau \riesz x_* - \big[ f'(x) + \big(H_x + \omega \riesz\big)(x_* - x)\big]}{X^*} \\ 
		&\leq \frac{\tau}{\tau + \kzwei} \norm{f'(x_*) - \big( f'(x) + (H_x + \omega \riesz)(x_* - x) \big)}{X^*} \\
		&= o\big( \norm{x_* - x}{X} \big) + \frac{\tau \, \omega}{\tau + \kzwei} \norm{x_* - x}{X}
	\end{align*}
	the last identity of which follows by the aforementioned definition of $H_x \in \partial_N f'(x)$ as a Newton-derivative together with \eqref{eq:semismooth}.
\end{proof}

\subsection{An Existing Inexactness Criterion}

In the literature composite gradient mappings have been used in order to derive an inexactness criterion for update step computation within Proximal Newton methods. 
%
Based on an approach from the smooth case, cf. e.g. \cite{Dembo1982}, the authors in \cite{Kanzow2020,Lee2014} took advantage of the composite gradient mapping $G^F_\tau$ to postulate the corresponding estimate which their inexact update steps have to satisfy. In a similar fashion, transferring 
the criterion from the smooth case to our globalization scheme using the damped update steps $\Delta s (\omega)$ from \eqref{eq:dampedstep} yields
\begin{align} \label{eq:gradcrit}
	\norm{G_{\tau}^{\hat F_{x,\omega}}\big( x + \Delta s(\omega) \big)}{X} \leq \eta \norm{G_{\tau}^F (x)}{X}
\end{align}
for some yet to be specified forcing term $\eta > 0$. Here, $\hat F_{x,\omega}$ denotes the modified quadratic model from \eqref{eq:dampedmodel} above. 
This requirement can be understood as a relative error criterion 
for the composite gradient mapping in norm due to the optimality of exactly computed update steps as formulated in Proposition~\ref{prop:gradprop}. 

While in a finite dimensional Euclidean space setting this gradient mapping can be evaluated efficiently due to the diagonal structure of the norm term, in an infinite dimensional setting the computation of this gradient mapping is quite demanding, even as expensive as computing the actual exact update step $\Delta x (\omega)$. 

Consequently, evaluating \eqref{eq:gradcrit} for every iteration within the subproblem solver becomes very costly and thereby immediately eclipses the savings we gain from inexactly computing the update steps. For this reason, we will resort to a different inexactness criterion.

\section{First Inexactness Criterion and Local Convergence Properties} \label{sec:local}

As pointed out beforehand, we do not use an inexactness criterion of the form \eqref{eq:gradcrit} due to its immense computational effort in function space. Instead, we exploit the advantageous properties of the TNNMG subproblem solver by resorting to an actual relative error estimate of the form
\begin{align} \label{eq:relerror}
\norm{\Delta x (\omega) - \Delta s (\omega)}{X} \leq \eta \norm{\Delta x (\omega)}{X}
\end{align}
where $\Delta x (\omega)$ denotes the exact solution of the update step computation subproblem \eqref{eq:dampedstep} and $\Delta s(\omega)$ is the corresponding inexact candidate. The influence of the forcing terms $\eta \geq 0$ on local convergence rates will be investigated in Theorem~\ref{thm:localconv}. 

Before actually stating the local convergence results, let us remark that the inexactness criterion \eqref{eq:relerror} is trivially satisfied by exactly computed update steps and $\eta$ is a measure for the margin for error which we allow in the computation. Additionally, the fact that the inexactly computed update steps $\Delta s (\omega)$ are in our case iterates from the convergent TNNMG subproblem solver implies that sooner or later within the solution process of \eqref{eq:dampedstep} the requirement \eqref{eq:relerror} will be satisfied.

Furthermore, let us comment on the efficient evaluation of this relative error estimate. At first sight, this is not completely obvious since apparently we do not have the exact solution $\Delta x (\omega)$ of the update computation subproblem \eqref{eq:dampedstep} at hand. In order to deal with this issue, we take advantage of the multigrid structure of the iterative subproblem solver which we employ, i.e., the TNNMG method from \cite{Graeser2018}. By $\delta^j$ we denote TNNMG-corrections, let therefore $\Delta s^i(\omega) = \sum_{j=1}^{i}\delta^j$ be an iterate within the inner solver towards the exact solution $\Delta x (\omega)$ and $\theta$ the 'constant' multigrid convergence rate for $\norm{\delta^j}{X} \leq \theta \norm{\delta^{j-1}}{X}$. Simple triangle inequalities thus provide us with
\begin{align*}
\norm{\Delta x (\omega) - \Delta s^i (\omega)}{X} = \sum_{j = i+1}^{\infty} \norm{\delta^j}{X} \leq \norm{\delta^i}{X} \sum_{j = i+1}^\infty \theta^{j-i} = \frac{\theta}{1-\theta} \norm{\delta^i}{X} \, .
\end{align*}
Similarly, for the norm of the exact solution we obtain
\begin{align*}
\norm{\Delta x (\omega)}{X} = \norm{\sum_{j=1}^\infty \delta^j}{X} &= \norm{\Delta s^i(\omega) + \sum_{j = i+1}^\infty \delta^j}{X} \geq \norm{\Delta s^i(\omega)}{X} - \norm{\sum_{j = i+1}^\infty \delta^j}{X} \\
&\geq \norm{\Delta s^i(\omega)}{X} - \frac{\theta}{1-\theta} \norm{\delta^i}{X} \, .
\end{align*}
Combining both of these estimates implies
\begin{align} \label{eq:altrelerror}
\frac{\norm{\Delta x (\omega) - \Delta s^i (\omega)}{X}}{\norm{\Delta x (\omega)}{X}} \leq \frac{\frac{\theta}{1 - \theta}\norm{\delta^i}{X}}{\norm{\Delta s^i(\omega)}{X} - \frac{\theta}{1 - \theta}\norm{\delta^i}{X}} \overset{!}{\leq} \eta
\end{align}
as a sufficient and easy to evaluate alternative inexactness criterion for the relative error estimate \eqref{eq:relerror}. Numerical experiments, which we also incorporated to Section~\ref{sec:numres}, clearly demonstrate that the performed triangle inequalities are sharper than one might have expected. Thus, the evaluation of the alternative criterion from \eqref{eq:altrelerror} comes very close to using the actual relative error for our computations later on.

Another crucial auxiliary result for all of the present treatise is an equivalence estimate between exactly computed update steps which have been damped according to different regularization parameters. It generalizes \cite[Lemma~6]{Poetzl2022} insofar that this result is comprised here in the case of $\omega = 0$.
\begin{lemma} \label{lem:equivomega}
	Let $\Delta x (\omega)$ and $\Delta x (\tilde \omega)$ be exactly computed update steps according to \eqref{eq:dampedstep} with regularization parameters satisfying $\omega > -(\ksum)$ and $\tilde \omega \geq \omega$. Then the following norm estimates hold:
	\begin{align}
		\norm{\Delta x (\omega) - \Delta x (\tilde \omega)}{X} &\leq \frac{\tilde \omega - \omega}{\omega + \ksum}\norm{\Delta x (\tilde \omega)}{X} \label{eq:diffnormest} \\
		\norm{\Delta x (\tilde \omega)}{X} \leq \norm{\Delta x (\omega)}{X} &\leq \frac{\tilde \omega + \ksum}{\omega + \ksum}\norm{\Delta x (\tilde \omega)}{X} \label{eq:equivomega}
	\end{align}
\end{lemma}
\begin{proof}
	The proof follows the exact same lines as the one of Lemma~6 in \cite{Poetzl2022} with the respective proximal representations of the exact steps used here.
\end{proof}

With the relative error inexactness criterion \eqref{eq:relerror} as well as the auxiliary results concerning regularized composite gradient mappings from Section~\ref{sec:gradientmap} and norm estimates from Lemma~\ref{lem:equivomega} at hand, we can now tackle the proof of the following local acceleration result.
\begin{theorem} \label{thm:localconv}
	Suppose that the semi-smoothness assumption \eqref{eq:semismooth} together with $\ksum > 0$ holds near an optimal solution $x^* \in X$ of \eqref{eq:prob}. Then, the inexact Proximal Newton method with update steps computed according to \eqref{eq:dampedstep} at $x_k \in X$ close to $x^*$ with the inexactness criterion \eqref{eq:relerror} for $\eta_k \geq 0$ exhibits the following local convergence behavior:
	\begin{itemize}
		\item[a)] The sequence of iterates locally converges linearly if $\omega_k$ and $\eta_k$ are sufficiently small, i.e., if there exists some constant $0 < \Theta < 1$ and $k_0 \in \NN$ such that for all $k \geq k_0$ the following estimate holds:
		\begin{align} \label{eq:linearconv}
			\frac{1}{\omega_k + \ksum}\big[ \big(\omega_k+ \norm{H_{x_k}}{\Lin} +\kzwei \big)\eta_k + \omega_k \big] < \Theta \, .
		\end{align}
		\item[b)] The sequence of iterates locally converges superlinearly in case both $\omega_k$ and $\eta_k$ converge to zero.
	\end{itemize}
\end{theorem}
\begin{proof}
	For the sake of simplicity, we will omit the sequence indices of all quantities here and denote $x = x_k$, $\omega = \omega_k$ and $\eta = \eta_k$ for the current iterate, regularization parameter and forcing term. For the next iterate, we write $x_+(\omega) = x_{k+1}(\omega)$ and $H_x = H_{x_k}$ stands for the current second order bilinear form. 
	
	For what follows, we fix $\tau \coloneqq \omega + \frac{1}{2}\big( \norm{H_x}{\Lin}+\keins \big)$ for the gradient mapping regularization parameter which allows us to take advantage of the auxiliary estimates deduced in Lemma~\ref{lem:GFxyz}. Under these circumstances, the first part of \eqref{eq:GFxyz} from Lemma~\ref{lem:GFxyz} provides us with
	\begin{align} \label{eq:localstart}
		\begin{split}
			\norm{x_+(\omega) - x_*}{X} &\leq \frac{1}{\tau \big( 1 - \mathcal H \big)}\norm{G_{\tau}^{\hat F_{x,\omega}}(x + \Delta s (\omega)) - G_{\tau}^{\hat F_{x,\omega}}(x_*)}{X} \\
			&\leq \frac{1}{\tau \big( 1 - \mathcal H \big)} \bigg[\norm{G_{\tau}^{\hat F_{x,\omega}}(x + \Delta s (\omega))}{X} + \norm{G_{\tau}^{\hat F_{x,\omega}}(x_*)}{X} \bigg]
		\end{split}
	\end{align}
	where we abbreviated the constant $\mathcal H \coloneqq \frac{\norm{H_x}{\Lin}-\keins}{2(\tau + \kzwei)}$. 
	As a next step, we take a look at the first norm term in brackets in \eqref{eq:localstart}. We use \eqref{eq:GFdeltaxzero} from Proposition~\ref{prop:gradprop} together with the second part of \eqref{eq:GFxyz} from Lemma~\ref{lem:GFxyz} for $y \coloneqq x+\Delta s (\omega)$ and $z \coloneqq x+ \Delta x (\omega)$ in order to obtain the following estimate:
	\begin{align*}
		\norm{G_\tau^{\hat F_{x,\omega}}\big( x + \Delta s (\omega) \big)}{X} &= \norm{G_\tau^{\hat F_{x,\omega}}\big( x + \Delta s (\omega) \big) - G_\tau^{\hat F_{x,\omega}}\big( x + \Delta x (\omega) \big)}{X} \\
		&\leq \tau \big( 1 + \mathcal H \big) \norm{\Delta x (\omega) - \Delta s (\omega)}{X} \, .
	\end{align*}
	For the ensuing norm difference we take advantage of the relative error estimate inexactness criterion \eqref{eq:relerror} together with the monotonicity of update step norms concerning the damping parameter $\omega$ as in Lemma~\ref{lem:equivomega}. Additionally, the superlinear convergence for full update steps close to optimal solutions (cf. \cite[Theorem~1]{Poetzl2022}) is important here:
	\begin{align} \label{eq:quanteta}
		\norm{\Delta x (\omega) - \Delta s (\omega)}{X} \leq \eta \norm{\Delta x(\omega)}{X} \leq \eta \norm{\Delta x}{X} \leq o\big( \norm{x-x_*}{X} \big) + \eta \norm{x-x_*}{X} \, .
	\end{align}
	By the optimality of $x_*$ together with Lemma~\ref{lem:GFGFk}, for the second term in brackets in \eqref{eq:localstart} we have
	\begin{align} \label{eq:quantomega}
		\norm{G_{\tau}^{\hat F_{x,\omega}}(x_*)}{X} = \norm{G_{\tau}^{\hat F_{x,\omega}}(x_*) - G_{\tau}^{F}(x_*)}{X} \leq o\big( \norm{x-x_*}{X} \big) + \frac{\omega \tau}{\tau + \kzwei}\norm{x-x_*}{X} \, .
	\end{align}
	The estimates \eqref{eq:quanteta} and \eqref{eq:quantomega} suffice to quantify the influence of either inexactness or damping on local convergence rates of our algorithm. Inserting both of them into \eqref{eq:localstart} above yields
	\begin{align} \label{eq:localfrac}
		\norm{x_+(\omega) - x_*}{X} &\leq \frac{(1 + \mathcal H)\eta + \frac{\omega}{\tau + \kzwei} }{1 - \mathcal H} \norm{x-x_*}{X} + o\big( \norm{x-x_*}{X} \big) \, .
	\end{align}
	All that remains to do now is simplify the rather complicated prefactor term within the estimate above. We expand the fraction by $2(\tau+\kzwei)$ and use that by the definition of $\tau$ we have
	\begin{align*}
		2(\tau+\kzwei) = 2(\omega + \kzwei) + \norm{H_x}{\Lin} + \keins \, .
	\end{align*}
	This provides us with
	\begin{align*}
		\frac{(1 + \mathcal H)\eta + \frac{\omega}{\tau + \kzwei} }{1 - \mathcal H} &= \frac{\big(2(\tau + \kzwei) + \norm{H_x}{\Lin}-\keins\big)\eta + 2\omega}{2(\tau + \kzwei) - \norm{H_x}{\Lin} + \keins} = \frac{\big( \omega + \norm{H_x}{\Lin} + \kzwei \big)\eta + \omega}{\omega + \ksum}
	\end{align*}
	Inserting this identity to \eqref{eq:localfrac} now directly yields
	\begin{align} \label{eq:localconv}
		\norm{x_+(\omega) - x_*}{X} \leq \frac{1}{\omega + \ksum} \big[ \big( \omega + \norm{H_x}{\Lin} + \kzwei \big)\eta + \omega \big] \norm{x-x_*}{X} + o\big( \norm{x-x_*}{X} \big) \, .
	\end{align}
	Now, both of the asserted cases for local convergence behavior are an immediate consequence of \eqref{eq:localconv}.
\end{proof}
\begin{remark}
    The estimate \eqref{eq:linearconv} yields a couple of algorithmically relevant insights.
    First, the linear convergence factor $\Theta$ can only be small, if both $\omega_k$ and $\eta_k$ are small. Hence, computing steps very accurately does only pay off, if $\omega_k$ is very small. We will see in Section~\ref{sec:transition} that close to optimal solutions arbitrarily small regularization parameters $\omega_k \approx 0$ can indeed be used.
    
    Second, if we neglect $\omega_k \approx 0$, then \eqref{eq:linearconv} simplifies to
%
	\begin{align*}
		\frac{\norm{H_{x_k}}{\Lin} + \kzwei}{\ksum} \eta_k \leq \Theta,
	\end{align*}
	where the prefactor on the left hand side can be interpreted as a local condition number of the problem. Indeed, for $\kzwei=0$ it coincides with the condition number of $H_x$ relative to $\|\cdot\|_X$. Thus, to achieve a given rate of local convergence, $\eta_k$ has to be chosen the tighter the higher the condition number. This underlines the necessity of an adequate choice of function space $X$ and norm $\|\cdot\|_X$. 
\end{remark}


Additionally, we were able to extend the local convergence result from \cite[Theorem~1]{Poetzl2022} insofar that we quantified the influence of damping update steps on (local) convergence rates. We are now also aware of more insightful criteria for linear or superlinear convergence of our method respectively. This helps us understand the process of local convergence of the (inexact) Proximal Newton method to an even greater extent.

\section{Global Convergence Properties} \label{sec:global}
Now that we have clarified the local convergence properties of our inexact Proximal Newton method depending on the forcing terms in criterion \eqref{eq:gradcrit}, we want to take into consideration whether the globalization scheme via the additional norm term in \eqref{eq:dampedstep} still fulfills its purpose and yields some global convergence results. 

\subsection{Cauchy Decrease Steps and the Subgradient Model} \label{subs:cauchy}

In order to achieve such a result, we will introduce a second crucial criterion which the inexactly computed update steps $\Delta s_k (\omega_k)$ have to satisfy in order to be admissible for our method. It can be viewed as an adopted strategy from smooth trust region methods where rather cheap so-called Cauchy decrease steps are used to measure functional value descent for the actual update steps, cf. e.g. \cite[Chapter~6]{Conn2000}.

There are several conceivable ways to define and compute such comparative Cauchy decrease steps. A canonical choice would be a simple Proximal Gradient step, i.e., the minimizer of the regularized linear model
\begin{align*}
\lambda_{x,\hat \omega}^C (\delta x) \coloneqq f'(x)\delta x + \frac{\hat \omega}{2}\norm{\delta x}{X}^2 + g(x + \delta x) - g(x) \, , \, \delta x \in X \, .
\end{align*}
As was the problem with evaluating the gradient mapping for our first inexactness criterion, also this procedure is as expensive as computing the exact Proximal Newton step right away in our general Hilbert space setting. 
Thus, the idea arises to find some comparative update step which we can compute with marginal effort in order to measure its functional value descent and then compare it to our inexact update step.

To this end, we define the subgradient model descent of $F$ around $x \in X$ with respect to $\mu \in \partial_F g(x)$ and regularization parameter $\hat \omega > 0$ by 
\begin{align} \label{eq:mumodel}
\lambda_{x,\hat \omega}^\mu (\delta x) \coloneqq f'(x)\delta x + \mu \, \delta x + \frac{\hat \omega}{2} \norm{\delta x}{X}^2 \, , \, \delta x \in X \, ,
\end{align}
and we refer to the respective minimizer  
\begin{align} \label{eq:mumin}
\Delta x^\mu (\hat \omega) \coloneqq \argmin_{\delta x \in X} \lambda_{x,\hat \omega}^\mu (\delta x)
\end{align}
as the corresponding subgradient step. Before introducing the second inexactness criterion which makes use of the above model and step, we will establish an analytical connection between \eqref{eq:mumodel} and our initially defined regularized second order decrease model $\lambda_{x,\omega}$ from \eqref{eq:dampedmodel}. To this end, we remember that the regularization parameter $\omega \geq 0$ is generally chosen such that the modified non-smooth part
\begin{align*}
\tilde g \colon X \to \RR \quad , \quad \tilde g (x) \coloneqq g(x) + \frac 12 \big( H_x + \omega \riesz \big)(x)^2
\end{align*}
is convex and thus the subproblem \eqref{eq:dampedstep} allows for a unique solution. Consequently, the characterization of the convex subdifferential $\partial \tilde g (x)$ yields that for any $\tilde \mu = \mu + (H_x + \omega \riesz) x \in \partial \tilde g (x)$ with $\mu \in \partial_F g (x)$ we have that
\begin{align*}
\tilde g (x + \delta x) \geq \tilde g (x) + \tilde \mu \, \delta x \quad \text{ and thus } \quad g(x + \delta x) - g(x) + \frac 12 H_x(\delta x)^2 + \frac{\omega}{2}\norm{\delta x}{X}^2 \geq \mu \, \delta x
\end{align*}
holds for any $\delta x \in X$ and  $\mu \in \partial_F g (x)$. We immediately obtain that
\begin{align} \label{eq:lambdamulambda}
\begin{split}
\lambda_{x,\hat \omega}^\mu (\delta x) &= f'(x)\delta x + \frac{\hat \omega}{2} \norm{\delta x}{X}^2 + \mu \delta x \\
&\leq f'(x)\delta x + \frac 12 H_x(\delta x)^2 + \frac{\hat \omega + \omega}{2}\norm{\delta x}{X}^2 + g(x + \delta x) - g(x) = \lambda_{x,\hat \omega + \omega} (\delta x)
\end{split}
\end{align}
is true for any $\delta x \in X$. In particular, this estimate apparently also holds for the respective minima of the decrease models of the composite objective function. For that reason, from \eqref{eq:lambdamulambda} we obtain
\begin{align} \label{eq:lambdamunorm}
\lambda_{x,\hat \omega}^\mu \big( \Delta x^\mu (\hat \omega) \big) \leq \lambda_{x,\hat \omega + \omega} \big(\Delta x (\hat \omega + \omega) \big) \leq - \frac 12 \big( \hat \omega + \omega + \ksum \big) \norm{\Delta x (\hat \omega + \omega)}{X}^2
\end{align}
for any $\hat \omega > 0$ where the last estimate constitutes a result from the exact case in \cite[Eq.(19)]{Poetzl2022} and will give us norm-like descent in the objective functional later on. Obviously, we now want to link this norm-like decrease within the subgradient model to the regularized second order decrease model $\lambda_{x,\omega} \big( \Delta s (\omega) \big)$ for our inexactly computed update step $\Delta s (\omega)$ and lastly to the direct descent within the objective functional $F$.

\subsection{Second Inexactness Criterion and Efficient Evaluation}

We will establish the first one of these connections via the actual second inexactness criterion which will thus also be checked within our algorithm and implementation. For this purpose, it is sufficient if an inexactly computed update step $\Delta s (\omega)$ satisfies the estimate
\begin{align} \label{eq:subgraddesc}
\lambda_{x,\omega} \big( \Delta s (\omega) \big) \leq \lambda_{x,\tilde \omega}^\mu \big( \Delta x^\mu (\tilde \omega) \big) \quad \text{for some} \quad \tilde \omega < \omax
\end{align} 
where the upper bound $\omax > 0$ is an algorithmic parameter yet to be specified. This inequality now constitutes our formal second inexactness criterion which we will also refer to as the \emph{subgradient inexactness criterion}.

Let us shortly elaborate on the efficient evaluation of this estimate and from there derive the actual implementation of the criterion: The solution property of $\Delta x^\mu (\tilde \omega)$ provides us with first order conditions for the corresponding minimization problem in the form of
\begin{align*}
0 = f'(x) + \mu + \hat \omega \riesz \Delta x^\mu (\tilde \omega)
\end{align*}
and thus $\Delta x^\mu (\tilde \omega) = - (\tilde \omega \riesz)^{-1} \big( f'(x) + \mu \big)$. For a given value of $\lambda_{x,\omega}\big( \Delta s (\omega) \big)$, i.e., descent within the regularized second order model with an inexactly computed update step, we can thus theoretically determine $\tilde \omega$ such that \eqref{eq:subgraddesc} is satisfied with equality. This can be seen as follows:
\begin{align} \label{eq:lambdaequality}
\begin{split}
\lambda_{x,\omega}\big( \Delta s (\omega) \big) &\overset{!}{=} \lambda_{x,\tilde \omega}^\mu \big( \Delta x^\mu (\tilde \omega) \big) = \big( f'(x) + \mu \big)\Delta x^\mu (\tilde \omega) + \frac{\tilde \omega}{2}\norm{\Delta x^\mu (\tilde \omega)}{X}^2 \\ 
&= \big( f'(x) + \mu \big) \big[  - (\tilde \omega \riesz)^{-1} \big( f'(x) + \mu \big)\big] + \frac{\tilde \omega}{2}\norm{- (\tilde \omega \riesz)^{-1} \big( f'(x) + \mu \big)}{X}^2 \\
&= - \frac{1}{2 \tilde \omega} \norm{f'(x) + \mu}{X^*}^2
\end{split}
\end{align}
which provides us with the theoretical value
\begin{align} \label{eq:omegatilde}
\tilde \omega = - \frac{\norm{f'(x) + \mu}{X^*}^2}{2 \lambda_{x,\omega} \big( \Delta s (\omega) \big)} \overset{!}{<} \omax
\end{align}
for the regularization parameter within the subgradient minimization problem \eqref{eq:mumin}. This quantity should remain bounded since otherwise the convergence of $\norm{\Delta x (\tilde \omega + \omega+1)}{X}^2$ to zero later will not provide us with global convergence results. Thus, as also pointed out in \eqref{eq:omegatilde}, we establish a sufficient estimate for our subgradient inexactness criterion \eqref{eq:subgraddesc} by demanding boundedness of $\tilde \omega$ from above by $\omax$. Note here that - as can be seen in \eqref{eq:lambdaequality} - the value for $\lambda_{x,\tilde \omega}^\mu \big( \Delta x^\mu (\tilde \omega) \big)$ increases as $\tilde \omega$ does.
Since globalization mechanisms in general should only provide worst case estimates and not slow down the convergence of our algorithm, we want the subgradient inexactness criterion to only interfere with update step computation on rare occasions and thus choose $\omax$ very large.

The dual norm occurring in the numerator of \eqref{eq:omegatilde} is computed as follows: we compute the minimizer of the linear subgradient model $\Delta x^\mu (1) \in X$ from \eqref{eq:mumin} and afterwards evaluate the linear functional $f'(x)+\mu \in X^*$ there. Here, the Fr\'echet-subdifferential element $\mu \in \partial_F g(x)$ is chosen such that the norm $\norm{f'(x) + \mu}{X^*}$ is as small as possible. Obviously, this depends on the specific minimization problem at hand but due to the non-smooth nature of $g$ it is often possible to exploit the set-valued subdifferential for this purpose.

Let us add some remarks concerning satisfiability of the subgradient inexactness criterion: As mentioned above, the freedom of choice of $\mu$ within $\partial_F g(x)$ opens up possibilities to decrease the value of $\norm{f'(x) + \mu}{X^*}$ right away. Additionally, considering the exact case for update step computation is very insightful in order to see that the criterion will be fulfilled by late iterations of the inner solver. For now, we interpret $\norm{f'(x) + \mu}{X^*} \approx \mathrm{dist}\big( \partial_F F (x) , 0 \big)$, i.e., we assume $\mu \in \partial_F g(x)$ to be chosen (nearly) optimally for our purpose of finding solutions of \eqref{eq:prob}.

\begin{proposition} \label{prop:secondinexsatis}
	Assume that there exists some constant $C > 0$ such that 
	\begin{align*}
	\norm{f'(x) + \mu}{X^*} \leq C \mathrm{dist}\big( \partial_F F (x) , 0 \big)
	\end{align*}
	holds at some $x \in X$ for $\mu \in \partial_F g(x)$. Then, the subgradient inexactness criterion \eqref{eq:subgraddesc} is eventually satisfied by iterates $\Delta s (\omega)$ of convergent solvers for the subproblem \eqref{eq:dampedstep} in case
	\begin{align} \label{eq:omaxbound}
	\omax > \frac{C^2\big(\omega + L_f + \norm{H_{x}}{\Lin}\big)^2}{\omega + \ksum}
	\end{align}
	holds for the upper bound $\omax$ from \eqref{eq:subgraddesc}.
\end{proposition}
\begin{proof}
	According to global convergence arguments in \cite[Theorem~2]{Poetzl2022} together with the assumed existence of $C>0$ above, we can estimate
	\begin{align*}
	\norm{f'(x) + \mu}{X^*} \leq C \mathrm{dist}\big( \partial_F F (x) , 0 \big) \leq C\big(L_f + \norm{H_{x}}{\Lin} + \omega\big)\norm{\Delta x (\omega)}{X}
	\end{align*}
	for the exactly computed update step $\Delta x (\omega)$. Additionally, from \cite[Eq.(19)]{Poetzl2022} we infer that
	\begin{align*}
	\lambda_{x,\omega}(\Delta x(\omega)) \leq - \frac12 \big( \omega +  \ksum \big)\norm{\Delta x(\omega)}{X}^2 \; \Leftrightarrow \; -2\lambda_{x,\omega}(\Delta x(\omega)) \geq \big( \omega +  \ksum \big)\norm{\Delta x(\omega)}{X}^2
	\end{align*}
	is true in this scenario and we consequently obtain
	\begin{align} \label{eq:exactomegatilde}
	\tilde \omega = - \frac{\norm{f'(x) + \mu}{X^*}^2}{2 \lambda_{x,\omega} \big( \Delta x (\omega) \big)} \leq \frac{C^2\big(\omega + L_f + \norm{H_{x}}{\Lin}\big)^2}{\omega + \ksum} < \infty \, .
	\end{align}
	Here, the convergence of the subproblem solver in the form that the respective objective value $\lambda_{x,\omega} \big( \Delta s (\omega)\big)$ tends to $\lambda_{x,\omega} \big( \Delta x (\omega)\big)$ from above comes into play. Thus, we can summarize
	\begin{align*}
	\tilde \omega = - \frac{\norm{f'(x) + \mu}{X^*}^2}{2 \lambda_{x,\omega} \big( \Delta s (\omega) \big)} \underset{>}{\to} - \frac{\norm{f'(x) + \mu}{X^*}^2}{2 \lambda_{s,\omega} \big( \Delta x (\omega) \big)} \leq \frac{C^2\big(\omega + L_f + \norm{H_{x}}{\Lin}\big)^2}{\omega + \ksum}
	\end{align*}
	for the theoretical value $\tilde \omega$ from \eqref{eq:omegatilde}. If now in particular the assumed estimate for the upper bound $\omax$ holds, the assertion directly follows.
\end{proof}
\begin{remark}
	The bound in \eqref{eq:omaxbound} in particular remains finite in both limits $\omega \to 0$ and $x \to \bar x$ for any stationary point $\bar x \in X$ of problem \eqref{eq:prob} near which  $\ksum > 0$ holds.
\end{remark}

The algorithmic strategy behind the subgradient inexactness criterion can now be summarized as follows: For the present iterate of the outer loop $x \in X$, we solve the linearized problem \eqref{eq:mumin} for the computation of the dual norm $\norm{f'(x) + \mu}{X^*}$ and initiate the inner loop in order to determine the next inexact update step.
At every iterate $\Delta s (\omega)$ of the inner solver for subproblem \eqref{eq:dampedstep} we compute the corresponding subgradient regularization parameter $\tilde \omega$ from \eqref{eq:omegatilde} and check $\tilde \omega < \omax$. As a consequence of Proposition~\ref{prop:secondinexsatis}, either $\omax$ is chosen large enough and we will eventually achieve $\tilde \omega < \omax$ for some inexact step or we will compute an exact update step $\Delta x (\omega)$ which on its own provides us with global convergence of the sequence of iterates as presented in \cite[Section~4]{Poetzl2022}.

\subsection{Summary of Inexactness Criteria}

With both of our inexactness criteria at hand, let us shortly reflect on their computational effort and compare it to possible alternatives: For the relative error criterion \eqref{eq:relerror} in its form \eqref{eq:altrelerror} only the evaluation of the fraction and its comparison to the forcing term is necessary since all occurring norms are already present within the subproblem solver. The subgradient inexactness criterion as described before requires the solution of the quadratic minimization problem \eqref{eq:mumin} once per outer iteration of our method together with the evaluation of the quadratic model $\lambda_{x,\omega}\big( \Delta s (\omega) \big)$ at each inner iteration which is a cheap operation.

For comparative algorithms from literature, cf. \cite{Kanzow2020,Byrd2015,Lee2014}, the gradient-like inexactness criterion \eqref{eq:gradcrit} has to be assessed at every inner iteration together with one comparison of the second order decrease model value with its base value for $\delta x = 0$. As mentioned before, the former operation is very costly for non-diagonal function space norms, particularly in comparison to solving a linearized problem once per outer iteration. This emphasizes both the necessity and the benefit of our adjustments to existing inexactness criteria. The summarized procedure can be retraced in the scheme of Algorithm~\ref{alg:inexproxnewton}.

\subsection{Sufficient Decrease Criterion and Global Convergence}

For global convergence in the case of inexactly computed update steps with the criteria introduced above we still have to carry out some more deliberations. The last missing ingredient in our recipe for norm-like descent within the composite objective functional is given by a sufficient decrease criterion which we have also used in the exact scenario in \cite[Eq.(18)]{Poetzl2022}. We say that an (inexactly computed) update step $\Delta s (\omega)$ is admissible for sufficient decrease if for some prescribed $\gamma \in ]0,1[$ the estimate
\begin{align} \label{eq:suffdes}
F\big(x+\Delta s(\omega)\big) - F(x) \leq \gamma \lambda_{x,\omega} \big( \Delta s (\omega) \big)
\end{align}
holds. Now, before justifying that \eqref{eq:suffdes} holds for sufficiently large values of the regularization parameter $\omega$, let us combine estimates \eqref{eq:suffdes}, \eqref{eq:subgraddesc}, the monotonicity of $\lambda_{x,\tilde \omega}^\mu \big( \Delta x^\mu (\tilde \omega) \big)$ with respect to $\tilde \omega$ as well as \eqref{eq:lambdamunorm} from above and thus recognize that we obtain
\begin{align} \label{eq:combined}
	\begin{split}
		F\big(x + \Delta s (\omega) \big) - F(x) &\leq \gamma \lambda_{x,\omega} \big( \Delta s (\omega) \big) =  \gamma \lambda_{x,\tilde \omega}^\mu \big( \Delta x^\mu (\tilde \omega) \big) \leq \gamma \lambda_{x,\tilde \omega + 1}^\mu \big( \Delta x^\mu (\tilde \omega + 1) \big) \\
		&\leq -\frac {(\tilde \omega + \omega + 1 + \ksum) \gamma}2 \norm{\Delta x (\tilde \omega + \omega + 1)}{X}^2 \\
		&\leq -\frac {\gamma} 2 \norm{\Delta x (\omax + \omega + 1)}{X}^2 \, .
	\end{split}
\end{align} 
Note that we additionally used $\tilde \omega \geq 0$ and $\omega + \ksum \geq 0$ as well as $\tilde \omega < \omax$ together with the equivalence result from Lemma~\ref{lem:equivomega}. 

The following lemma ensures the satisfiability of the sufficient decrease criterion \eqref{eq:suffdes} as soon as $\omega$ is large enough.
\begin{lemma} \label{lem:suffdesomega}
	Let $\omega$ denote the regularization parameter ensuring the unique solvability of the update step subproblem \eqref{eq:dampedstep}. If $\omega$ is sufficiently increased to some $\omega_+ > \omega$, the sufficient decrease criterion \eqref{eq:suffdes} is fulfilled by inexactly computed update steps $\Delta s (\omega_+)$ which additionally satisfy the inexactness criteria \eqref{eq:relerror} and \eqref{eq:subgraddesc}.
\end{lemma}
\begin{proof}
	The first inexactness criterion \eqref{eq:relerror} provides us with the norm estimate
	\begin{align} \label{eq:skleinerx}
	\norm{\Delta s (\omega_+)}{X} \leq \norm{\Delta s (\omega_+) - \Delta x (\omega_+)}{X} + \norm{\Delta x (\omega_+)}{X} \leq (1+\eta) \norm{\Delta x (\omega_+)}{X}
	\end{align}
	such that similar to \eqref{eq:combined} we obtain
	\begin{align*}
	\lambda_{x,\omega_+} \big( \Delta s (\omega_+) \big) \leq -\frac {1} 2 \norm{\Delta x (\omax + \omega + 1)}{X}^2 \leq -\frac {1} 2 \norm{\Delta x (\omega_+)}{X}^2 \leq - \frac{1}{2(1+\eta)^2} \norm{\Delta s (\omega_+)}{X}^2
	\end{align*}
	eventually for $\omega_+ > \omax + \omega + 1$ by the equivalence result in Lemma~\ref{lem:equivomega}.
	
	Consequently, we can assume that for sufficiently large $\omega_+$, an estimate of the form
	\begin{align*}
	\lambda_{x,\omega_+} \big(\Delta s(\omega_+)\big) \leq - c\norm{\Delta s (\omega_+)}{X}^2
	\end{align*}
	holds also for the inexactly computed update steps and some constant $c > 0$. From here, we can employ the same proof as for \cite[Lemma~3]{Poetzl2022} and conclude the assertion.
\end{proof}
\begin{remark}
	The above result together with the assumption \eqref{eq:kappa1} and \eqref{eq:kappa2} on our objective functional also imply that the regularization parameter $\omega$ remains bounded over the course of the minimization process.
\end{remark} 
Let us now deduce the ensuing global convergence results for the inexact Proximal Newton method as presented in the scheme of Algorithm~\ref{alg:inexproxnewton}. 

\begin{algorithm}
	\caption{Inexact second order semi-smooth Proximal Newton Algorithm} \label{alg:inexproxnewton}
	\label{alg:buildtree}
	\begin{algorithmic}
		\REQUIRE Starting point $x_0 \in X$, sufficient decrease parameter $\gamma \in ]0,1[$, initial values $\omega_0$ and $\eta_0$, $\varepsilon > 0$ for stopping criterion
		\STATE{Initialization: $k=0$}
		\WHILE{$(1+\omega_k)\norm{\Delta s_k(\omega_k)}{X} \geq \varepsilon$}
		\STATE{Choose $\mu \in \partial_Fg(x_k)$ and compute norm term for $\tilde \omega$ as in \eqref{eq:omegatilde} via the linearized minimization problem \eqref{eq:mumin}}
		\STATE{Compute a trial step $\Delta s_k (\omega_k)$ according to \eqref{eq:dampedstep} which suffices the inexactness criteria \eqref{eq:altrelerror} and \eqref{eq:omegatilde}}
		\WHILE{Sufficient decrease criterion \eqref{eq:suffdes} is not satisfied}
		\STATE{Increase $\omega_k$ appropriately}
		\STATE{Recompute $\Delta s_k (\omega_k)$ as above}
		\ENDWHILE
		\STATE{Update current iterate to $x_{k+1} \leftarrow x_k + \Delta s_k(\omega_k)$}
		\STATE{Decrease $\omega_{k}$ to some $\omega_{k+1} < \omega_k$ for next iteration}
		\STATE{Decrease $\eta_k$ to some $\eta_{k+1} < \eta_k$ for next iteration}
		\STATE{Update $k \leftarrow k+1$}
		\ENDWHILE
		\RETURN $x$
	\end{algorithmic}
\end{algorithm}

For this reason, we will first prove that the right-hand side of \eqref{eq:combined}, i.e., the norm of exactly computed comparative steps $\Delta x (\omax + \omega + 1)$, converges to zero along the sequence of iterates generated by inexact updates. Here, it will come in handy to define $\omega^c \coloneqq \omax + \omega + 1$ for the regularization parameter of the comparative exact update steps. Note that this quantity is bounded both from above and below.
\begin{lemma} \label{lem:comptozero}
	Let $(x_k) \subset X$ be the sequence generated by the inexact Proximal Newton method globalized via \eqref{eq:dampedstep} starting at any $x_0 \in \mathrm{dom}g$. Additionally, suppose that the subgradient inexactness criterion \eqref{eq:subgraddesc} and the sufficient decrease criterion \eqref{eq:suffdes} are satisfied. Then either $F(x_k) \to -\infty$ or $\norm{\Delta x_k(\omega^c)}{X} \to 0$ for $k \to \infty$.
\end{lemma}
\begin{proof}
	By \eqref{eq:combined} the sequence $F(x_k)$ is monotonically decreasing. Thus, either $F(x_k) \to -\infty$ or $F(x_k) \to \underline F$ for some $\underline F\in \RR$ and thereby in particular $F(x_{k})-F(x_{k+1})\to 0$. As a consequence of \eqref{eq:combined}, then also $\norm{\Delta x_k(\omega^c)}{X} \to 0$ holds.
\end{proof}
Note that the above result does not comprise the convergence of the sequence of iterates itself which is desirable in the context. In the exact case of update step computation it was possible to take advantage of first order optimality conditions of the exactly solved subproblem for the actual update steps and from there achieve a proper global convergence result at least in the strongly convex case, cf. \cite[Theorem~3]{Poetzl2022}. Due to the presence of inexactness in update step computation this strategy has to be slightly adjusted in the current scenario, i.e., applied to the comparative update steps $\Delta x (\omega^c)$. To this end, for some $k \in \NN$ and iterate $x_k \in X$ we introduce the so-called corresponding comparative iterate
\begin{align} \label{eq:comparative}
y_k \coloneqq x_k + \Delta x_k (\omega^c) = \p_g^{H_{x_k} + \omega^c \riesz} \big( (H_{x_k} + \omega^c \riesz)x_k - f'(x_k) \big) \, .
\end{align}
Note here that the comparative iterate uses a theoretical exact update but origins at the iterate $x_k$ which belongs to our inexact method. Also, for every $k \in \NN$ the identity $y_k - x_k = \Delta x_k(\omega^c)$ holds by definition of $y_k$.

With this definition at hand, we are in the position to discuss at least subsequential convergence of our algorithm to a stationary point. In the following, we will assume throughout that the sequence of objective values $\big(F(x_k)\big)$ is bounded from below. We start with the case of convergence in norm:
\begin{theorem} \label{thm:strongstat}
	Assume that the subgradient inexactness criterion \eqref{eq:subgraddesc} and the sufficient decrease criterion \eqref{eq:suffdes} are fulfilled. Then, all accumulation points $\bar x$ (in norm) of the sequence of iterates $(x_k)$ generated by the inexact Proximal Newton method globalized via \eqref{eq:dampedstep} are stationary points of problem \eqref{eq:prob}. In particular, the comparative sequence $(y_k)$ defined via \eqref{eq:comparative} satisfies 
	\begin{align*}
	\mathrm{dist}\big( \partial_F F(y_k),0 \big) \to 0 \quad \text{and} \quad \norm{x_k - y_k}{X} \to 0 \, ,
	\end{align*}
	i.e., also $y_k \to \bar x$ for $k \to \infty$.
\end{theorem}
\begin{proof}
	By $(x_k)$ we denote the subsequence of iterates converging to the accumulation point $\bar x$. As mentioned beforehand, for the corresponding comparative sequence $(y_k)$ we have $y_k - x_k = \Delta x_k(\omega^c)$ and consequently also $y_k \to \bar x$ holds by $\norm{\Delta x_k (\omega^c)}{X} \to 0$ due to Lemma~\ref{lem:comptozero}. The proximal representation of $y_k$ in \eqref{eq:comparative} is equivalent to the minimization problem
	\begin{align*}
	y_k = \argmin_{y \in X} g(y) + \frac 12 \big( H_{x_k} + \omega^c \riesz \big)(y)^2 - \big( (H_{x_k} + \omega^c \riesz)x_k - f'(x_k) \big)y
	\end{align*}
	which yields the first order optimality conditions given by the dual space inclusion
	\begin{align*}
	0 \in \partial_F g(y_k) + f'(x_k) + \big( H_{x_k} + \omega^c \riesz \big)(y_k - x_k) \, .
	\end{align*}
	This, on the other hand, is equivalent to
	\begin{align} \label{eq:ohneQ}
	\big( H_{x_k} + \omega^c \riesz \big) (x_k - y_k) + f'(y_k) - f'(x_k) \in \partial_F g (y_k) + f'(y_k) = \partial_F F(y_k)
	\end{align}
	the remainder term on the left-hand side of which we can estimate via
	\begin{align*}
	\norm{\big( H_{x_k} + \omega^c \riesz \big)(x_k - y_k) + f'(y_k) - f'(x_k)}{X^*} &\leq \big( M + \omega^c + L_f \big)\norm{x_k - y_k}{X} \\ &= \big( M + \omega^c + L_f \big)\norm{\Delta x_k^c(\omega^c)}{X} \to 0
	\end{align*}
	for $k \to \infty$ where $M$ denotes the uniform bound on the second order bilinear form norms from assumption \eqref{eq:Hxbound}.
	
	In order to now achieve the optimality assertion of the accumulation point $\bar x$, we have to slightly adjust \eqref{eq:ohneQ} for the use of the convex subdifferential and its direct characterization. To this end, we consider a bilinear form $Q:X \times X \to \RR$ such that the function $\tilde g: X \to \RR$ defined via $\tilde g(x) \coloneqq g(x) + \frac 12 Q(x)^2$, $x\in X$, is convex. As above, $Q\coloneqq H_{x_k} + \omega_k\riesz$ is a reasonable choice. Inserting a $Q(y_k)$-term into \eqref{eq:ohneQ} thus yields
	\begin{align*}
	\omega^c \riesz (x_k - y_k) + f'(y_k) - f'(x_k) \in \partial \tilde g (y_k) + \{f'(y_k) - Q(y_k)\}
	\end{align*}
	for the convex subdifferential of $\tilde g$. The left-hand side now as before converges to zero in $X^*$ and consequently, we know that for every $k\in \NN$ there exists some $\tilde \rho_k \in \partial \tilde g (y_k)$ such that we can define $\tilde \rho \coloneqq \lim_{k \to \infty} \tilde \rho_k = -f'(\bar x) + Q \bar x$ by the convergence of also $y_k$ to $\bar x$. The lower semi-continuity of $\tilde g$ together with the definition of the convex subdifferential $\partial \tilde g$ directly yields
	\begin{align*}
	\tilde g (u) - \tilde g(\bar x) &= \tilde g(u)- g(\bar x) - \frac 12 Q(\bar x)^2 \geq \tilde g(u) - \liminf_{k \to \infty} g(y_k) - \lim_{k \to \infty} \frac 12 Q(y_k)^2 \\
	&= \liminf_{k \to \infty} \tilde g(u) - \tilde g(y_k) \geq \liminf_{k \to \infty} \tilde \rho_k (u- y_k) = \lim_{k \to \infty} \tilde \rho_k (u- y_k) = \tilde \rho (u-\bar x)
	\end{align*}
	for any $u \in X$ which proves the inclusion $\tilde \rho \in \partial \tilde g(\bar x)$. The evaluation of the latter limit expression can easily be retraced by splitting
	\begin{align} \label{eq:weaklim}
	\tilde \rho_k (u- y_k) = \tilde \rho_k (u-\bar x) + (\tilde \rho_k - \tilde \rho)(\bar x- y_k) + \tilde \rho (\bar x- y_k) \, .
	\end{align}
	In particular, we recognize $\tilde \rho \in \partial \tilde g(\bar x)$ as $-f'(\bar x) + Q \bar x \in \partial \tilde g(\bar x)$ and equivalently $-f'(\bar x) \in \partial_F g(\bar x)$ for the Frech\'et-subdifferential $\partial_F$. This implies $0 \in \partial_F F (\bar x)$, i.e., the stationarity of our accumulation point $\bar x$.
\end{proof}
We remember from the exact case in \cite{Poetzl2022} that we can indeed interpret $\norm{\Delta x_k(\omega_k)}{X} \leq \varepsilon$ for some small $\varepsilon > 0$ as a condition for the optimality of the current iterate up to some prescribed accuracy. Estimate \eqref{eq:skleinerx} from above thus yields that also the norm of the inexactly computed update steps can be used as an optimality measure for the current iterate within our method.
However, small step norms $\norm{\Delta s_k(\omega_k)}{X}$ can also occur due to very large values of the damping parameter $\omega_k$ as a consequence of which the algorithm would stop even though the sequence of iterates is not even close to an optimal solution of the problem. In order to rule out this inconvenient case, we consider the scaled version $(1+\omega_k)\norm{\Delta s_k(\omega_k)}{X}$ as the stopping criterion in the later implementations of our algorithm.


Let us now proceed to generalizing the convergence result from Theorem~\ref{thm:strongstat}: While bounded sequences in finite dimensional spaces always have convergent subsequences, we can only expect \emph{weak subsequential convergence} in general Hilbert spaces in this case. As one consequence, existence of minimizers of non-convex functions on Hilbert spaces can usually only be established in the presence of some compactness. On this count, we note that in \eqref{eq:weaklim} even weak convergence of $x_k \rightharpoonup \bar x$ would be sufficient. Unfortunately, in the latter case we cannot evaluate $f'(x_k) \to f'(\bar x)$. 
In order to extend our proof to this situation, we require some more structure for both of the parts of our composite objective functional. The proof is completely analogous to the one of \cite[Theorem~3]{Poetzl2022}.
\begin{theorem}
	Let $f$ be of the form $f(x)= \hat f(x) + \check f(Kx)$ where $K$ is a compact operator. Additionally, assume that $g + \hat f$ is convex and weakly lower semi-continuous in a neighborhood of stationary points of \eqref{eq:prob}. Then weak convergence of the sequence of iterates $x_k \rightharpoonup \bar x$ suffices for $\bar x$ to be a stationary point of \eqref{eq:prob}.
	
	If $F$ is strictly convex and radially unbounded, the whole sequence $(x_k)$ converges weakly to the unique minimizer $x_*$ of $F$. If $F$ is $\kappa$-strongly convex, with $\kappa > 0$, then $x_k \to x_*$ in norm. 
\end{theorem}

\section{Transition to Local Convergence} \label{sec:transition}

In order to now benefit from the local acceleration result in Theorem~\ref{thm:localconv}, we have to manage the transition from the globalization phase above to the local convergence phase described beforehand. To this end, we have to make sure that (at least close to optimal solutions of \eqref{eq:prob}) arbitrarily small regularization parameters $\omega \geq 0$ yield update steps that give us sufficient decrease in $F$ according to the criterion formulated in \eqref{eq:suffdes}. This endeavor has also been part of the investigation of the exact case in \cite[Section~6]{Poetzl2022} but as for all aspects of our convergence analysis has to be slightly adapted here.

As a starting point, a rather technical auxiliary result is required. It sets the limit behavior of inexact update steps in relation with the distance of consecutive iterates to the minimizer of \eqref{eq:prob}. 
\begin{lemma} \label{lem:oestimates}
	Let $x$ and $x_+ (\omega)=x+\Delta s (\omega)$ be two consecutive iterates with update step $\Delta s (\omega)$ sufficing \eqref{eq:relerror} for some $0 \leq \eta < 1$. Furthermore, consider an optimal solution $x_*$ of \eqref{eq:prob}. Then the following estimates eventually hold for $\ksum > 0\colon$
	\begin{align*}
	\norm{x_+(\omega) - x_*}{X} \leq (3 + \eta)\norm{x-x_*}{X} \quad , \quad \norm{x-x_*}{X} \leq \frac{2}{1-\eta}\big(1+\frac{\omega}{\ksum}\big)\norm{\Delta s (\omega)}{X} \, .
	\end{align*}
\end{lemma}
\begin{proof}
	Our proof here mainly exploits the local superlinear convergence of exactly computed and undamped update steps $\Delta x$ from \cite[Theorem~1]{Poetzl2022} and then uses the respective estimates in order to introduce the influences of both damping and inexactness. For the first asserted estimate, we take a look at
	\begin{align*}
	\norm{x_+(\omega) - x_*}{X} &\leq \norm{x-x_*}{X} + \norm{\Delta s (\omega)}{X} \leq \norm{x-x_*}{X} + (1+\eta)\norm{\Delta x}{X} \\
	&\leq 2\norm{x-x_*}{X} + (1+\eta)\norm{x + \Delta x - x_*}{X}
	\end{align*}
	where the second step involved \eqref{eq:skleinerx} together with $\norm{\Delta x (\omega)}{X} \leq \norm{\Delta x}{X}$ as proven in Lemma~\ref{lem:equivomega}. From here, we use the superlinear convergence of exact updates in the form of the existence of some function $\psi:[0,\infty[ \to [0,\infty[$ with $\psi(t) \to 0$ for $t \to 0$ such that
	\begin{align*}
	\norm{x + \Delta x - x_*}{X} = \psi\big( \norm{x-x_*}{X} \big) \norm{x-x_*}{X}
	\end{align*}
	holds in the limit of $x \to x_*$. Thus, we obtain 
	\begin{align*}
	\norm{x_+(\omega) - x_*}{X} \leq \big[ 2 + (1+\eta)\psi\big( \norm{x-x_*}{X} \big) \big] \norm{x-x_*}{X} \leq (3+\eta)\norm{x-x_*}{X}
	\end{align*}
	since eventually we can assume the $\psi$-term to be smaller than one. This completes the proof of the first asserted estimate.
	
	For the second one we take advantage of
	\begin{align*}
	\norm{\Delta x}{X} \leq \big( 1 + \frac \omega {\ksum} \big)\norm{\Delta x (\omega)}{X}
	\end{align*}
	from Lemma~\ref{lem:equivomega} together with again the superlinear convergence as above and find that
	\begin{align*}
	\norm{x-x_*}{X} &\leq \norm{x + \Delta x - x_*}{X} + \norm{\Delta x}{X} \leq \psi\big( \norm{x-x_*}{X} \big) \norm{x-x_*}{X} + \big( 1 + \frac \omega {\ksum} \big)\norm{\Delta x (\omega)}{X}
	\end{align*}
	holds. Since the $\psi$-term eventually will be smaller than $\frac 12$, from here we infer
	\begin{align*}
	\norm{x-x_*}{X} \leq \frac{1 + \frac \omega {\ksum}}{1 - \psi\big( \norm{x-x_*}{X} \big)}\norm{\Delta x (\omega)}{X} \leq 2\big( 1 + \frac \omega {\ksum} \big)\norm{\Delta x (\omega)}{X} \, .
	\end{align*}
	The inexactness of update step computation now enters the above estimate using the inequality $\norm{\Delta x (\omega)}{X} \leq \frac{1}{1-\eta} \norm{\Delta s (\omega)}{X}$ which can easily be retraced via
	\begin{align*}
		(1-\eta)\norm{\Delta x (\omega)}{X} &\leq \norm{\Delta x (\omega)}{X} - \norm{\Delta x (\omega) - \Delta s (\omega)}{X} \\
		&\leq \norm{\Delta x (\omega) -\big( \Delta x (\omega) - \Delta s (\omega) \big)}{X} = \norm{\Delta s (\omega)}{X}
	\end{align*}
	with the inexactness criterion \eqref{eq:relerror}. This completes the proof of the lemma. 
\end{proof}
\begin{remark}
In particular, these eventual norm estimates have implications on the limit behavior of the respective terms. If we now have $\xi = o\big(\norm{x_+(\omega) - x_*}{X}\big)$ for some $\xi \in X$, $\xi = o\big(\norm{x - x_*}{X}\big)$ immediately holds and from there we obtain $\xi = o\big(\norm{\Delta s (\omega)}{X}\big)$ in the same way.
\end{remark}

In what follows, it will be important several times that the second order bilinear forms $H_x$ satisfy a bound of the form
\begin{align}  \label{eq:Hcontass}
\big(H_{x_+(\omega)}-H_x\big) \big(x_+(\omega)-x_*\big)^2 = o\big( \norm{x-x_*}{X}^2 \big) \, \mbox{ for } x \to x_* \, .
\end{align}
It is easy to see that the bound holds if either we have uniform boundedness of the second order bilinear forms together with superlinear convergence of the iterates or if we have continuity of the mapping $x \mapsto H_x$ together with mere convergence of the iterates to $x_*$. Note here that the same assumption has been made in the exact case in \cite{Poetzl2022} for the admissibility of undamped and arbitrarily weakly damped update steps. In our scenario, we conclude that according to Theorem~\ref{thm:localconv} it is sufficient that both the regularization parameters $\omega_k \geq 0$ and the forcing terms $\eta_k \geq 0$ converge to zero as we approach the optimal solution $x^* \in X$ of \eqref{eq:prob} together with assumption \eqref{eq:Hxbound} from the introductory section. We will later on establish this convergence of $(\omega_k)$ and $(\eta_k)$ in the specific implementation of our algorithm.

With the auxiliary estimates from Lemma~\ref{lem:oestimates} and Lemma~\ref{lem:equivomega} together with the thoroughly discussed additional assumption from \eqref{eq:Hcontass} at hand, we can now turn our attention to the actual admissibility of arbitrarily small update steps close to optimal solutions of \eqref{eq:prob}.

For that matter, we furthermore suppose $f$ to be second order semi-smooth at optimal solutions $x_*$ of \eqref{eq:prob} with respect to the mapping $H:X \to \Lin , x \mapsto H_x$, which expresses itself via the estimate
\begin{align} \label{eq:sos}
f(x_* + \xi) = f(x_*) + f'(x_*)\xi + \frac 12 H_{x_* + \xi}(\xi,\xi) + o\big(\norm{\xi}{X}^2\big) \quad \text{for} \quad \norm{\xi}{X} \to 0 \, .
\end{align}
This notion generalizes second order differentiability in our setting but its definition slightly differs from semi-smoothness of $f'$ as qualified in \eqref{eq:semismooth}. For further elaborations on this concept of differentiability, consider \cite[Section~5]{Poetzl2022}.
\begin{proposition} \label{prop:dampedstepadmiss}
	Suppose that the additional assumptions \eqref{eq:Hcontass} and \eqref{eq:sos} hold. Furthermore, assume that the update steps $\Delta s (\omega)$ computed as inexact solutions of \eqref{eq:dampedstep} at $x\in X$ for some $\omega \geq 0$ satisfy the inexactness criteria \eqref{eq:relerror} for $\eta \geq 0$ and \eqref{eq:subgraddesc}. Then, $\Delta s (\omega)$ is admissible for sufficient decrease according to \eqref{eq:suffdes} for any $\gamma < 1$ if $x$ is sufficiently close to an optimal solution $x^* \in X$ of \eqref{eq:prob} where $\ksum > 0$ holds.
\end{proposition}
\begin{proof}
	We take a look back at the proof of \cite[Proposition~8]{Poetzl2022} and employ the same telescoping strategy in order to obtain
	\begin{align*}
	&f(x_+(\omega)) - f(x) - f'(x)\Delta s(\omega) - \frac 12 H_x (\Delta s(\omega))^2 \\
	&=\left[f(x_+(\omega))-f(x_*)-f'(x_*)(x_+(\omega)-x_*)-\frac12 H_{x_+(\omega)}(x_+(\omega)-x_*)^2\right] \\
	&\quad -\left[f(x) - f(x_*) - f'(x_*)(x-x_*) - \frac12 H_x(x-x_*)^2\right]  \\
	&\quad -\Big[(f'(x)-f'(x_*))\Delta s(\omega)-H_x(x-x_*,\Delta s(\omega))\Big]+ \frac12(H_{x_+(\omega)}-H_x) (x_+(\omega)-x_*)^2
	\end{align*}
	where again we can use the second order semi-smoothness of $f$ according to \eqref{eq:sos} for the first two terms as well as the semi-smoothness of $f'$ as in \eqref{eq:semismooth} for the third one. This implies
	\begin{align*}
	f(x_+(\omega)) - f(x) - f'(x)\Delta s(\omega) - \frac 12 H_x (\Delta s(\omega))^2 &= o(\|x_+(\omega)-x_*\|^2)+o(\|x-x_*\|_X^2) \\
	&\quad +o(\|x-x_*\|_X)\|\Delta s(\omega)\|_X + \rho(x,\omega)
	\end{align*}
	where we denoted $\rho(x,\omega) \coloneqq \frac12(H_{x_+(\omega)}-H_x) (x_+(\omega)-x_*)^2$. Due to the limit behavior of inexact update step norms investigated over the course of Lemma~\ref{lem:oestimates} this yields
	\begin{align} \label{eq:rhoundkleino}
	f(x+\Delta s(\omega)) - f(x) - f'(x)\Delta s (\omega) - \frac 12 H_x\big(\Delta s(\omega)\big)^2 = \rho(x,\omega) + o\big(\norm{\Delta s (\omega)}{X}^2\big) \, .
	\end{align}
	As the next step towards the admissibility result, we define the prefactor function
	\begin{align*}
	\gamma (x,\omega) \coloneqq \frac{F\big(x + \Delta s (\omega)\big)- F(x)}{\lambda_{x,\omega}\big(\Delta s (\omega)\big)}
	\end{align*}
	which should be larger than some $\tilde \gamma \in ]0,1[$ for $\Delta s (\omega)$ to yield sufficient decrease according to \eqref{eq:suffdes}. Thus, it suffices to show the convergence of $\gamma (x,\omega)$ to anything greater equal than one for any $\omega \geq 0$ in the limit of $x \to x^*$. The identity \eqref{eq:rhoundkleino} from above now provides us with 
	\begin{align*}
	F\big(x + \Delta s (\omega)\big)- F(x) = \lambda_{x,\omega}(\Delta s(\omega))-\frac{\omega}{2}\|\Delta s(\omega)\|^2_X + \rho(x,\omega) + o\big(\|\Delta s(\omega)\|_X^2\big)
	\end{align*}
	which we insert into the prefactor function from above and estimate
	\begin{align} 
	\begin{split} \label{eq:gammaest}
	\gamma(x,\omega) &= 1+\frac{-\frac{\omega}{2}\|\Delta s(\omega)\|^2_X 
		+\rho(x,\omega)+o(\|\Delta s(\omega)\|_X^2) }{\lambda_{x,\omega}\big(\Delta s(\omega)\big)}\\
	&=1+\frac{\frac{\omega}{2}\|\Delta s(\omega)\|^2_X 
		+o(\|\Delta s(\omega)\|_X^2)-\rho(x,\omega)}{ |\lambda_{x,\omega}\big(\Delta s(\omega)\big)|}
	\end{split}
	\end{align}
	since from the computation strategy for $\Delta s (\omega)$ we in particular have
	\begin{align} \label{eq:negativitylambdas}
	\lambda_{x,\omega}\big(\Delta s(\omega)\big) \leq  \lambda_{x,\tilde \omega}^\mu \big( \Delta x^\mu (\tilde \omega) \big) \leq -\frac 1 2 \norm{\Delta x (\omega^c)}{X}^2 \leq 0
	\end{align}
	following the later steps of \eqref{eq:combined}. For the absolute value of the second order decrease model we can use \eqref{eq:negativitylambdas} together with Lemma~\ref{lem:equivomega} and \eqref{eq:skleinerx} to obtain
	\begin{align}
	\begin{split} \label{eq:lambdaabs}
	|\lambda_{x,\omega}\big(\Delta s(\omega)\big)| &\geq |\lambda_{x,\tilde \omega}^\mu \big( \Delta x^\mu (\tilde \omega) \big)| \geq \frac 1 2 \norm{\Delta x (\omega^c)}{X}^2 \geq \frac 1 2 \big( \frac{\omega + \ksum}{\omega^c + \ksum} \big)^2 \norm{\Delta x (\omega)}{X}^2 \\
	&\geq \frac 1 2 \big( \frac{\omega + \ksum}{(1+\eta)(\omega^c + \ksum)} \big)^2 \norm{\Delta s (\omega)}{X}^2 \eqqcolon C \norm{\Delta s (\omega)}{X}^2
	\end{split}
	\end{align}
	where $C = C(\omega,\omega^c,\keins+\kzwei,\eta)>0$ denotes the constant from above. In particular, note that $C$ remains bounded in the limit of $\omega \to 0$ and is also well-defined in the limit case of $\omega = 0$ close to optimal solutions with $\ksum > 0$.
	
	We may assume that the numerator of the latter expression in \eqref{eq:gammaest} is non-positive, otherwise the desired inequality for $\gamma (x,\omega)$ is trivially fulfilled. Thus, we take advantage of \eqref{eq:lambdaabs} in order to decrease the positive denominator to achieve
	\begin{align*}
	\gamma(x,\omega) \geq 1 + \frac{\omega}{2C} - \varepsilon - \frac{\rho(x,\omega)}{C\norm{\Delta s (\omega)}{X}^2}
	\end{align*}
	where for any $\varepsilon > 0$ there exists a neighborhood of the optimal solution $x_*$ such that the above estimate holds.
	
	Now, the assumption \eqref{eq:Hcontass} for the $\rho$-term immediately implies the eventual admissibility of $\Delta s (\omega)$ for sufficient decrease according to \eqref{eq:suffdes}.
\end{proof}

\section{Numerical Results} \label{sec:numres}
Let us now showcase the functionality of our inexact Proximal Newton method and also compare its performance to the case of exact computation of update steps. To this end, we consider the following function space problem on $\Omega \coloneqq [0,1]^3 \subset \RR^3$: Find a vector field $u \in H^1_0(\Omega,\RR^3)$ that minimizes the composite objective functional $F$ defined via
\begin{align} \label{eq:toymodelfunc}
\begin{split}
F(u) \coloneqq f(u) +  \int_\Omega c \norm{u}{1} \, \de \omega
\end{split}
\end{align}
for some parameter $c > 0$ as a weight for the $\mathrm{L}_1$-norm term where the smooth part $f\colon H^1_0(\Omega,\RR^3) \to \RR$ is given by
\begin{align*}
	f(u) \coloneqq \int_\Omega \frac 12 \norm{\nabla u}{F}^2 + \alpha \max\big( \norm{\nabla u}{F} - 1, 0 \big)^2 + \beta \frac{u_1 ^3 u^2_2 u_3}{1+u_1^2+u_2^2+u_3^2} + \rho \cdot u \, \de \omega
\end{align*}
with parameters $\alpha,\beta \in \RR$ as well as a force field $\rho : \Omega \to \RR^3$. The norm $\norm{\cdot}{F}$ denotes the Frobenius norm of the respective Jacobian matrices $\nabla u$. 

We have to note here that $f$ technically does not satisfy the assumptions made on the smooth part of the composite objective functional specified above in the case $\alpha \neq 0$ due to the lack of semi-smoothness of the corresponding squared max-term. The use of the derivative $\nabla u$ instead of function values $u$ creates a norm-gap which cannot be, as usual,  compensated by Sobolev-embeddings and hinders the proof of semi-smoothness of the respective superposition operator. However, we think that slightly going beyond the framework of theoretical results for numerical investigations can be instructive. 

In what follows we will choose the force-field $\rho$ to be constant on $\Omega$ and to this end introduce the so-called load factor $\tilde \rho > 0$ which then determines $\rho = \tilde \rho (1,1,1)^T$. Now that we have fully prescribed the composite objective functional $F$, we recognize that its non-smooth part $g$ is given by the integrated $\mathrm{L}_1$-norm term with constant prefactor $c > 0$. 
Let us also emphasize here that the underlying Hilbert space is given by $X = H^1_0(\Omega,\mathbb R^3)$ which also determines the norm choice for regularization of the subproblem.

Now, we will explain the specifics of our implementation of the method: 
%
To compute inexact update steps via the second order model problem \eqref{eq:dampedstep} used a so-called Truncated Non-smooth Newton Multigrid Method (TNNMG): In short, this method can be described as a mixture of exact, non-smooth Gauß-Seidel steps for each block component and global truncated inexact Newton steps enhanced with a line-search procedure. Analytical proofs for convergence for convex and coercitive problems as well as  convergence properties have been established, cf. \cite{Graeser2018}, and functionality for demanding applications has been investigated, cf. \cite{Graeser2020} and \cite{Sander2019}. Additionally, this subproblem solver is provided with stopping criteria in the form of our inexactness criteria \eqref{eq:altrelerror} and \eqref{eq:omegatilde} with corresponding parameters $\eta_k \in [0,1]$ for each iteration and global $\omax>0$. The required derivatives were computed by automatic differentiation, using adol-C, cf. \cite{Walther2012}. 

Another topic of interest concerning the implementation of our algorithm is the choice of the aforementioned parameters $\omega$, $\eta$ and $\omax$ governing the convergence behavior of our method. While - as discussed in its introduction in \eqref{eq:omegatilde} - $\omax$ can be chosen constant and is supposed to be very large, this is not the case for the regularization parameters $\omega$ and the forcing terms $\eta$. Adaptive choices for these quantities are subject to our current research and yield promising results but in the present treatise we want to focus on the aspect of and criteria for inexactness itself. Thus, we decided to take a rather heuristic approach of doubling $\omega$ in the case that the update step was not accepted and multiply it by $\big( \frac 12 \big)^{n\cdot n}$ (where $n$ denotes the number of successful consecutive updates) in the remaining case. Once $\omega$ drops below some threshold value, we set it to zero in order to locally use undamped update steps.

Similarly, we multiply the forcing term $\eta$ by $0.6$ for accepted updates and leave it as it is in case the increment was rejected by the sufficient decrease criterion. This rather simple strategy for the choice of parameters ensures the convergence of both $\eta$ and $\omega$ to zero along the sequence of iterates and thus also from a theoretical standpoint enables superlinear convergence as formulated in Theorem~\ref{thm:localconv}.
In addition to the correction norm stopping criterion for the outer loop in Algorithm~\ref{alg:inexproxnewton}, we introduced a threshold value for the descent according to the modified quadratic model $\lambda_{x,\omega}\big( \Delta s (\omega) \big)$, i.e., the computation stops once we achieve $(1+\omega)\big|\lambda_{x,\omega}\big( \Delta s (\omega) \big)\big| < 10^{-13}$ for an admissible step $\Delta s (\omega)$.

Let us now consider the actual tests we performed in order to demonstrate the performance of our algorithm: Firstly, we will demonstrate the consistency between results of the inexact method and the exact version the functionality of which has been thoroughly investigated in \cite{Poetzl2022}. Afterwards, we exhibit the gains in effectivity by enhancing the exact algorithm with the inexactness criteria introduced above. Lastly, we analyze the implementation of the latter criteria and try to get a grasp on how they affect the process of solving the subproblem for update step computation. 

All in all, we use \eqref{eq:toymodelfunc} with fixed parameters $c = 80$, $\beta = 40$, $\tilde \rho = -100$ and let $\alpha \geq 0$ vary. Increasing $\alpha$ magnifies the influence of the squared $\max$-term in \eqref{eq:toymodelfunc} and thus makes the corresponding minimization problem harder to solve.

For the first one of the above concerns, we consider Figures \ref{fig:normcorr5lvl} and \ref{fig:energies5lvl} which display plots for either the $H^1$-norms of (accepted) update steps or the energies (, i.e., objective values,) at the corresponding updated iterates for values of $\alpha$ from 40 to 160 in steps of 40.

\begin{figure}[h]
	\centering
	\begin{subfigure}[b]{0.45\textwidth}
		\begin{center}
			\resizebox{\textwidth}{!}{\input{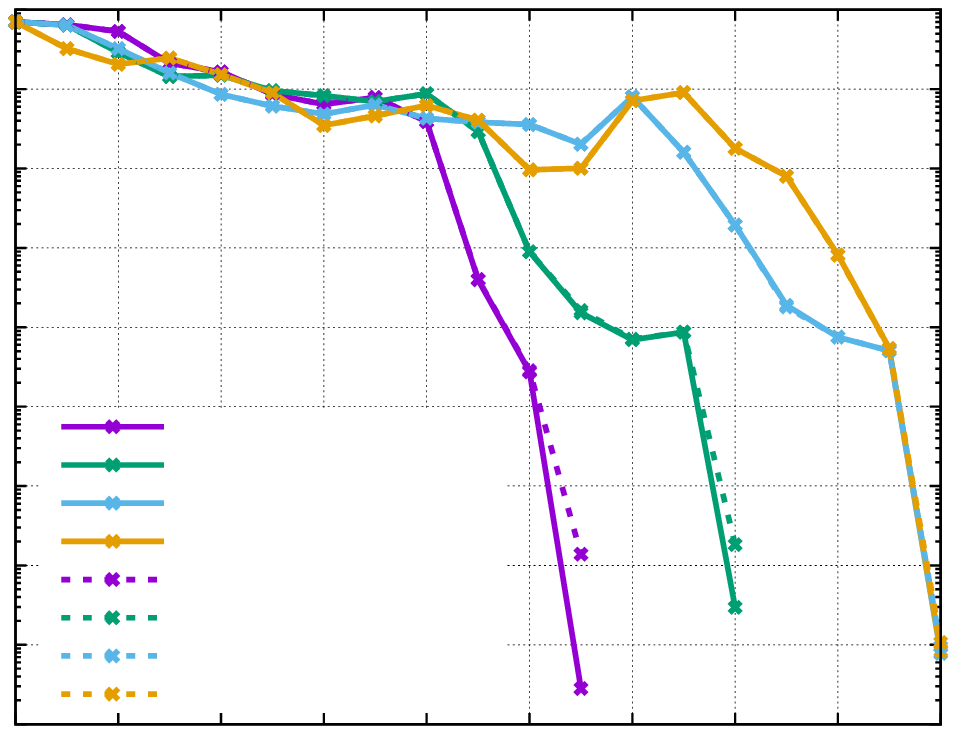}}
		\end{center}
		\caption{Correction norms $\norm{\Delta s_k (\omega_k)}{H^1(\Omega)}$}
		\label{fig:normcorr5lvl}
	\end{subfigure}
	\hfill
	\begin{subfigure}[b]{0.45\textwidth}
		\begin{center}
			\resizebox{\textwidth}{!}{\input{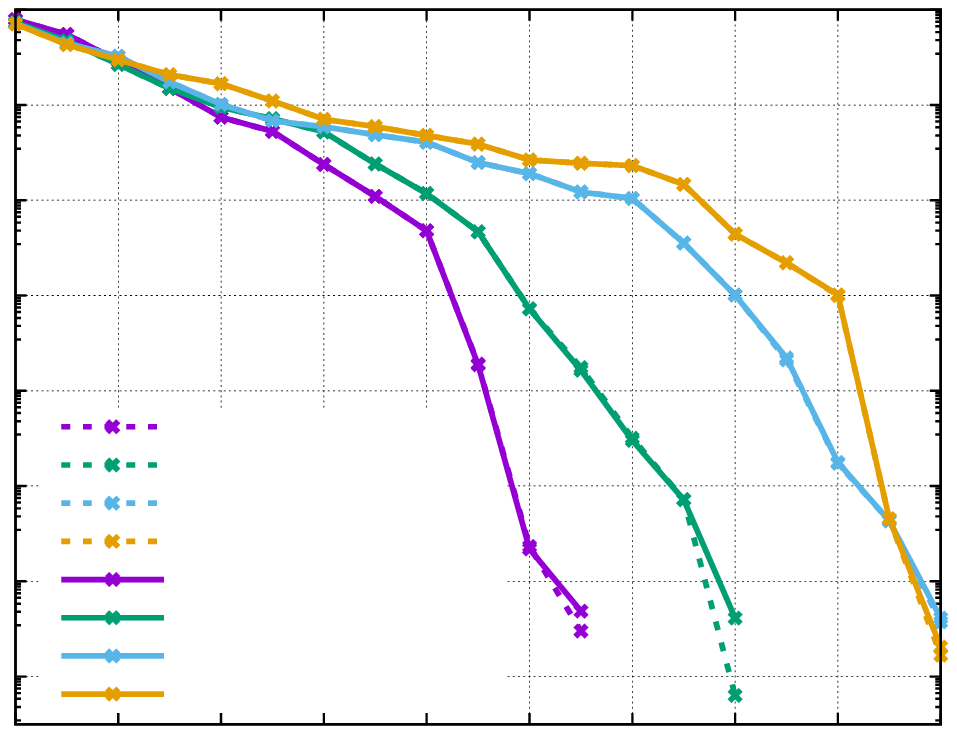}}
		\end{center}
		\caption{Objective function values $F(x_k)$}
		\label{fig:energies5lvl}
	\end{subfigure}
	\caption{Graphs of correction norms and corresponding objective values for $c = 80$, $\beta = 40$, $\tilde \rho = -100$ and $\alpha \in \{40,80,120,160\}$ for the Proximal Newton method with five uniform grid refinements.}
\end{figure}

As the plots illustrate, the difference in quality of the update steps for exact and inexact computations is marginal to non-existent. This can also be retraced in Table \ref{tab:totaliterations} where we list the number of accepted iterations and such which did not yield sufficient decrease according to \eqref{eq:suffdes} for $\gamma = \frac 12$.

\begin{table}[h]
	\begin{center}
		\begin{tabular}{| c | c || c | c | c | c | c | c | c | c | c | c | c |}
			\hline
			\multicolumn{2}{|c||}{$\alpha$} 								& 0 & 	20 & 	40 & 	60 & 	80 & 	100 & 	120 & 	140 & 	160 & 	180 & 	200 \\
			\hline 
			\xrowht{12pt} \multirow{2}{*}{Exact\vspace{-5mm}} 	& Accepted & 15 &	8 & 	12 & 	15 & 	15 & 	12 & 	19 & 	23 & 	19 & 	22 & 	21 \\
			\cline{2-13}
			\xrowht{12pt} 										& Declined & 11 & 	5 & 	10 & 	17 & 	19 & 	13 & 	23 & 	68 & 	25 & 	33 & 	30 \\
			\hline
			\xrowht{12pt} \multirow{2}{*}{Inexact\vspace{-5mm}} & Accepted & 15 & 	9 & 	12 & 	14 & 	15 & 	14 & 	19 & 	20 & 	19 & 	20 & 	21 \\
			\cline{2-13}
			\xrowht{12pt}									    & Declined & 7 & 	4 & 	10 & 	18 & 	21 & 	44 & 	23 & 	29 & 	25 & 	31 & 	30 \\
			\hline
		\end{tabular}
	\end{center}
	\caption{Number of accepted and declined iterations for different prefactor values $\alpha$ for fixed parameters $\beta=40$, $c=80$ and $\tilde \rho = -100$ in the exact and inexact case.}
	\label{tab:totaliterations}
\end{table}

Let us now consider the number of subproblem solver iterations we saved by allows inexact computations. Figure~\ref{fig:tnnmgiterations} displays the number of TNNMG-iterations necessary for the computation of each accepted update step.
\begin{figure}[h]
	\begin{center}
		\resizebox{0.5\textwidth}{!}{\input{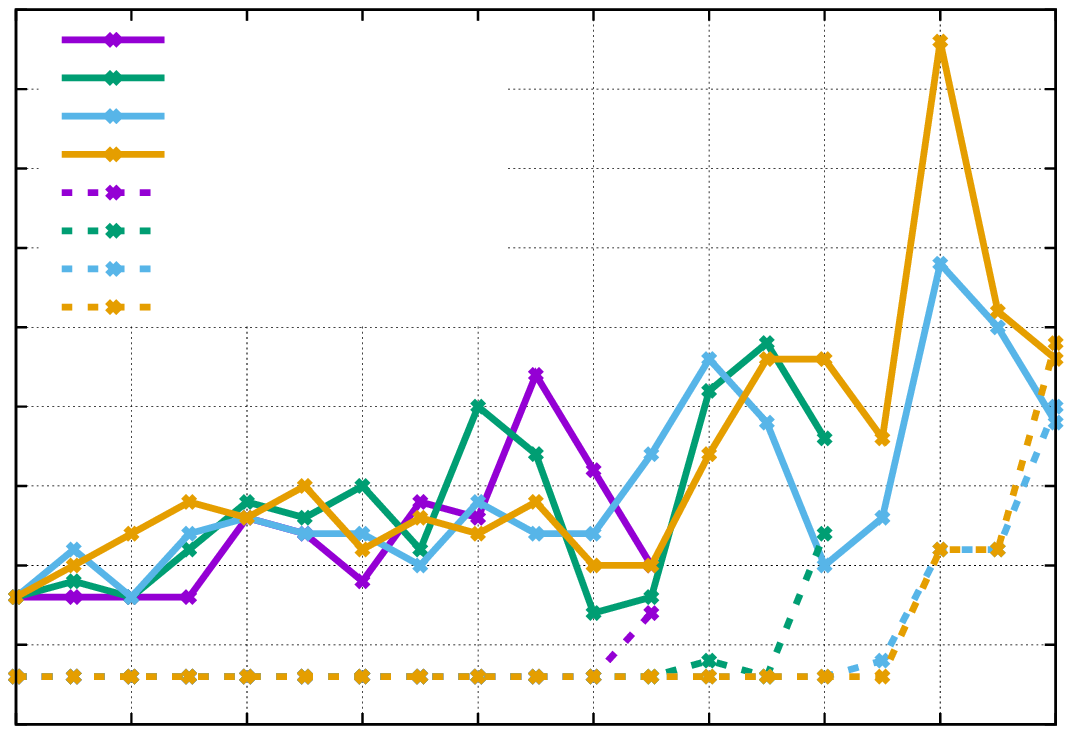}}		
	\end{center}
	\caption{Comparison of the number of TNNMG-iterations for $c = 80$, $\beta = 40$, $\tilde \rho = -100$ and $\alpha \in \{40,80,120,160\}$ for the exact and inexact case.}
	\label{fig:tnnmgiterations}
\end{figure}
Also here, the results leave no room for interpretation and substantiate the effectivity of the inexact method. To further reinforce these findings, we give the number of total TNNMG-iterations for the scenarios from Table \ref{tab:totaliterations} in Table \ref{tab:tnnmgiterations}. Note here that these numbers include the subproblem steps for the computation of both accepted and declined updates $\Delta s(\omega)$.
\begin{table}[h]
	\begin{center}
		\begin{tabular}{| c || c | c | c | c | c | c | c | c | c | c | c |}
			\hline
			\xrowht{12pt} $\alpha$  	& 0 	& 20 	& 40 	& 60 	& 80 	& 100 	& 120 	& 140 	& 160 	& 180 	& 200 \\
			\hline
			\xrowht{12pt} Exact 		& 227 	& 143 	& 269 	& 434 	& 488 	& 405 	& 695 	& 943 	& 784 	& 983 	& 954 \\
			\hline
			\xrowht{12pt} Inexact 		& 66 	& 40 	& 70 	& 108 	& 143 	& 182 	& 160 	& 211 	& 169 	& 202 	& 220 \\
			\hline
		\end{tabular}
	\end{center}
	\caption{Comparison of the number of total TNNMG-iterations $N$ for the exact and the inexact case in the scenario of Table~\ref{tab:totaliterations}.}
	\label{tab:tnnmgiterations}
\end{table}
We can see that we at least spare two thirds of the steps within the subproblem solver and as $\alpha$ increases even only need a quarter of them in comparison to the exact method.

As mentioned beforehand, we also want to take a look at how the inexactness criteria affect the solving process of the step computation subproblems. To this end, we consider two aspects each of which covers one of our criteria based on an exemplary Proximal Newton steps: On the one hand, in order to investigate the relative error criterion \eqref{eq:relerror}, we computed every Proximal Newton step twice. Within the first computation, we neglected inexactness criteria which allowed us to then compute the actual relative error $E_{rel}$ of the TNNMG iterates in the second and actually inexact computation process. This makes it possible to compare the relative error to the estimate $E_{est}$ which we use for easier evaluation, cf. \eqref{eq:altrelerror}. As can be seen in 
the plots in Figure~\ref{fig:relerror}, both quantities stay within the same order of magnitude and eventually drop below the bound $\eta$ from the inexactness criterion. This implies that the estimate which implicitely uses the convergence rate of our multigrid subproblem solver constitutes and adequate and easy-to-evaluate alternative to the actual relative error. Note that the estimated error $E_{est}$ is not assigned within the first two TNNMG iterations since we have to take more of these into consideration in order to obtain a valid estimate for multigrid convergence rates $\theta$ in \eqref{eq:altrelerror}.

On the other hand, we also considered the subgradient inexactness criterion \eqref{eq:subgraddesc}. As mentioned beforehand, we introduced this criterion for globalization purposes with the intention that it would not interfere with the minimization process, especially in the local acceleration phase close to optimal solutions. In fact, we have noticed that throughout our tests the determining quantity for further solving the subproblem was the relative error estimate and not that $\tilde \omega$ from \eqref{eq:omegatilde} was too large. For example, over the TNNMG-iterations of the Proximal Newton step considered in Figure~\ref{fig:relerror} we had nearly constant $\tilde \omega \approx 1.5$, clearly remaining below our choice of $\omax \coloneqq 10^8$.

\begin{figure}[h]
	\begin{center}
		\resizebox{0.5\textwidth}{!}{\input{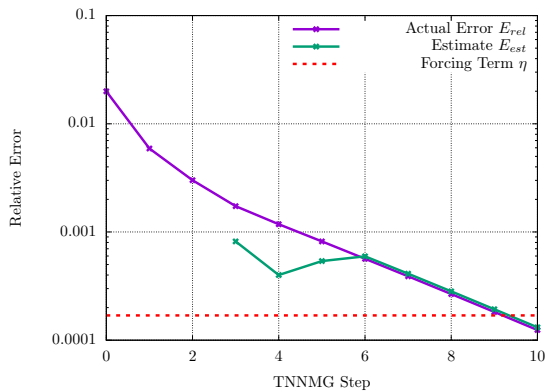}}
	\end{center}
	\caption{Comparison of the relative error $E_{rel}$ and its estimate $E_{est}$ together with the forcing term bound $\eta$ in Proximal Newton step $k = 17$ for $\alpha = 160$.}
	\label{fig:relerror}
\end{figure}

\section{Conclusion} \label{sec:concl}


We have extended the globally convergent and locally accelerated Proximal Newton method in Hilbert spaces from \cite{Poetzl2022} to inexact computation of update steps. Additionally, we have improved local convergence proofs by considering regularized gradient mappings and have thereby disclosed the influence of damping and inexactness to local convergence rates. We have found inexactness criteria that suit the general infinite-dimensional Hilbert space setting of the present treatise and can be evaluated cheaply within every iteration of the subproblem solver. Using these inexactness criteria, we have also been able to carry over all convergence results, local as well as global, from the exact case. The application of our method to actual function space problems is enabled by using an efficient solver for the step computation subproblem, the Truncated Non-smooth Newton Multigrid Method. We have displayed functionality and efficiency of our algorithm by considering a simple model problem in function space.

 Room for improvement is definitely present in the choice of both regularization parameters $\omega$ and forcing terms $\eta$. The former can be addressed by different approaches like estimates for residual terms of the quadratic model established in subproblem \eqref{eq:dampedstep}, cf. \cite{Weiser2007}, or adapted strategies for controlling time step sizes in computing solutions of ordinary differential equations. For the forcing terms on the other hand, adaptive choices have already been studied for inexact Newton methods e.g. in \cite{An2007,DdBook}. While these can be carried over to our non-smooth scenario, it also appears to be promising to tie the choice of regularization parameters and forcing terms together due to their similar convergence behavior. This idea both reduces the computational effort and better reflects the problem structure at hand.


\bibliographystyle{siamplain}
\bibliography{References.bib}
\end{document}